\documentclass[a4]{article}

\usepackage{pb-diagram}
\usepackage{hyperref}
\usepackage{amsmath,amsxtra,amssymb,mathrsfs,verbatim,amscd}
\usepackage{amsfonts,stmaryrd,graphicx}
\usepackage{dsfont}
\usepackage{amsthm}
\usepackage[all]{xy}

\newtheorem{thm}{Theorem}[section]
\newtheorem{cor}[thm]{Corollary}
\newtheorem{lem}[thm]{Lemma}
\newtheorem{prop}[thm]{Proposition}

\theoremstyle{definition}
\newtheorem{dfn}[thm]{Definition}
\newtheorem{rmk}[thm]{Remark}
\newtheorem{exm}[thm]{Example}

\newcommand{\ilim}[1][]{\mathop{\varinjlim}\limits_{#1}}

\newcommand{\tf}{{\mathcal{T}}}
\newcommand{\pl}{{\underset{W\in\mathcal{W}}{\oplus}\mathscr{F}(W)}}
\newcommand{\hmx}{{{\rm H}^m_M(X,\mathscr{F})\underset{\mathbb{Z}_M(M)}{\otimes}or_{M/X}(M)}}

\newcommand{\rcsp}{\mathbb D_{\mathbb{C}^n}}
\newcommand{\hexpo}{{\mathscr{O}}^{\operatorname{exp}}_{\rcsp}}

\newcommand{\bexpo}{{\mathscr{B}}^{\operatorname{exp}}_{\mathbb{D}_{\mathbb R}^n}}

\begin{document}

\title{Intuitive representation of local cohomology groups}

\date{}

\author{Daichi Komori and Kohei Umeta}


\maketitle

\begin{abstract}
\qquad We construct a framework which gives intuitive representation of local cohomology groups.
By defining the concrete mappings among them, we show their equivalence.
As an application, we justify intuitive representation of Laplace hyperfunctions.
\end{abstract}

\section{Introduction}

\qquad The theory of hyperfunctions was established by M.~Sato in 1950s \cite{S},\cite{SKK}.
He defined hyperfunctions by applying the local cohomology functor to the sheaf of holomorphic functions and it was a starting point of algebraic analysis and microlocal analysis.

The definition of a hyperfunction depends on several general theories such as the sheaf theory, its cohomology theory and so on,  which are not necessarily common among people who study analysis.
To understand a hyperfunction without these heavy theories, A.~Kaneko \cite{Kane} and M.~Morimoto \cite{Mori} defined hyperfunctions in an intuitive way.
Their idea is as follows;
Let $M$ be an open subset in $\mathbb{R}^n$ and $X$ its complex neighborhood.
For the sheaf $\mathscr{O}$ of holomorphic functions, they realize a hyperfunction $h\in\mathscr{B}_M(M)$ as a formal sum of holomorphic functions $f_W\in\mathscr{O}(W)$, where $W$ is an infinitisimal wedge of type $M\times\sqrt{-1}\Gamma$ with an $\mathbb{R}_+$-conic open subset $\Gamma$, that is,
\begin{equation}
\displaystyle h(x)=\sum_{W\in\mathcal{W}} f_W(x+\sqrt{-1}y),
\end{equation}
where $\mathcal{W}$ is a family of infinitisimal wedges with edge $M$.
Their definition allows us to manipulate a hyperfunction like a function as it is represented by holomorphic functions on wedges.

Their idea can be applied to not only the theory of hyperfunctions but also local cohomology groups $\hmx$,
where $X$ is a topological manifold, $M$ is a closed submanifold of $X$ of codimension $m$ and $\mathscr{F}$ is a sheaf on $X$.
The aim of this paper is to construct intuitive representation $\hat{\rm H}(\mathscr{F}(\mathcal{W}))$
of local cohomology groups,
and show the equivalence of them.
To justify their intuitive representation,
we introduce a framework which consists of a triplet $(\mathcal{T},\mathcal{U},\mathcal{W})$ satisfying several conditions.
Here $\mathcal{T}$ is a family of stratifications of $S^{m-1}$,
$\mathcal{U}$ is a family of open neighborhoods of $M$ and
$\mathcal{W}$ is a family of open sets in $X$.
In this framework we define intuitive representation $\hat{\rm H}(\mathscr{F}(\mathcal{W}))$ of local cohomology groups by
\begin{equation}
\hat{\rm H}(\mathscr{F}(\mathcal{W}))=\left(\pl\right)/\mathcal{R},
\end{equation}
where $\mathcal{R}$ is a $\mathbb{C}$-vector space so that two sections which coincide on their non-empty common domain give the same element in $\hat{\rm H}(\mathscr{F}(\mathcal{W}))$.

We shall show the equivalence of local cohomology groups and their intuitive representation by constructing the isomorphisms concretely.
In this course, we can construct the boundary value morphism $b_\mathcal{W}$ from their intuitive representation to local cohomology groups in a functorial way due to Schapira's idea \cite{KS}.
The inverse map $\rho$ of the boundary value morphism is, however, not obvious.
Therefore by introducing an particular resolution which depends on the choices of a stratification $\chi\in\mathcal{T}$, we obtain the expression of local cohomology groups depending on $\chi$.
Then we realize the inverse map $\rho$ thanks to this expression.

As an application, for example, we can construct a framework which realizes intuitive representation of Laplace hyperfunctions.
The theory of Laplace hyperfunctions in one variable was established by H.~Komatsu as a framework of operational calculus, and is extended to the one in several variables by K.~Umeta and N.~Honda \cite{Honda}.
In particular, they established a vanishing theorem of cohomology groups on a pseudoconvex open subset for holomorphic functions with exponential growth at infinity.
Their result is the extension of Oka-Cartan's vanishing theorem to $\mathbb{D}_{\mathbb{C}^n}$, and is a crucial key for constructing a framework in the application.
Here $\mathbb{D}_{\mathbb{C}^n}$ is the radial compactification of $\mathbb{C}^n$.
In addition, we establish a theorem of Grauert type which guarantees the existence of a good open neighborhood in our sense,
from which we can conclude the existence of the framework for Laplace hyperfunctions.

\

At the end of the introduction, we would like to show our greatest appreciation to Professor Naofumi Honda for the valuable advice and generous support in Hokkaido University.

\section{Intuitive Representation}
In this section, we introduce several definitions which are needed for constructing a framework.
Then we define intuitive representation of local cohomology groups within this framework.

\subsection{Preparation}

Let $X$ be a topological manifold and $M$ a closed submanifold of $X$.
Note that these manifolds are allowed to have their boundaries.
Let $\mathscr{F}$ be a sheaf on $X$, $\mathbb{Z}_X$ the constant sheaf on $X$ having stalk $\mathbb{Z}$ and $m$ a non negative integer.
We denote by ${\rm H}^m_M(X,\mathscr{F})$ the $m$-th local cohomology group of $X$ supported by $M$.
We also denote by $D^m$ the open unit ball in $\mathbb{R}^m$ with the center at the origin and set $S^{m-1}=\partial{D^m}$.
Then we define $\widetilde{X}$ as follows.
\begin{eqnarray}{\label{xm}}
\widetilde{X}=M\times D^m.
\end{eqnarray}
We assume that there exists a homeomorphism $\iota: X\rightarrow \widetilde{X}$ which satisfies $\iota(M)=M\times\{0\}$.
Hereafter we identify $X$ with $\widetilde{X}$ by $\iota$.

As the first step of constructing framework, we introduce an appropriate partition of $S^{m-1}$.
\begin{dfn}
Let $k=0,1,\cdots,m-1$.
A simplex $\sigma$ is said to be a linear $k$-cell in $S^{m-1}$ if $\sigma$ is oriented and it is written in the form
\begin{equation}
\sigma=\varphi(\cap H_i)\cap S^{m-1}.
\end{equation}
Here $\varphi : \mathbb{R}^{k+1} \to \mathbb{R}^m$ is a linear injective map and $\{H_i\}$ is a finite family of open half spaces in $\mathbb{R}^{k+1}$ with $0\in \partial H_i$ whose intersection is non-empty.
For convenience, we regard a empty set $\emptyset$ as a linear cell.
\end{dfn}

Let us recall the definition of a stratification of a set $U$.
A partition $U=\underset{\lambda\in\Lambda}{\sqcup}U_\lambda$ is called a stratification of $U$, if it is locally finite, and it satisfies for all the pairs $(U_\lambda,U_\tau)$
\begin{center}
$\overline{U_\lambda}\cap U_\tau\neq \emptyset\Rightarrow U_\tau\subset\overline{U_\lambda}$,
\end{center}
and we call each $U_\lambda$ a stratum.
Let $|\sigma|$ denote the realization,
i.e., just the set in $S^{m-1}$ forgetting its orientation, of the simplex $\sigma$.

Let $\chi$ be a stratification of $S^{m-1}$ such that each stratum is a linear cell of $S^{m-1}$.
We denote by $\Delta(\chi)$ and $\Delta_k(\chi)$ the set of all linear cells of $\chi$ and the set of all linear $k$-cells of $\chi$, respectively.
We also denote by $|\Delta|(\chi)$ the set of realizations $|\sigma|$ of cells $\sigma \in \Delta(\chi)$, that is,
$$
|\Delta|(\chi) = \{|\sigma|\,;\, \sigma \in \Delta(\chi)\}.
$$
The set $|\Delta|_k(\chi)$ is also defined in the same way.
\begin{dfn}

For two stratifications $\chi$ and $\chi'$ of $S^{m-1}$, $\chi'$ is finer than $\chi$ if and only if for any $\sigma'\in\Delta(\chi')$ there exists $\sigma\in\Delta(\chi)$ with $\sigma'\subset\sigma$, which is denoted by $\chi\prec\chi'$.
\end{dfn}

\begin{dfn}
 For a cell $\sigma\in\Delta(\chi)$ we define a star open set ${\rm St}_\chi(|\sigma|)$ as follows.
\begin{equation}
{\rm St}_\chi(|\sigma|)=\underset{\substack{\sigma\subset\overline{\tau} \\ \tau\in\Delta(\chi)}}{\sqcup}|\tau|.
\end{equation}
\end{dfn}
As a special case, we set ${\rm St}_\chi(|\emptyset|)=S^{m-1}$.
We write ${\rm St}(|\sigma|)$ instead of ${\rm St}_\chi(|\sigma|)$ if there is no confusion.

\begin{dfn}
Let $\chi$ be a stratification of $S^{m-1}$ whose stratum is a linear cell
and let $\sigma,\tau\in\Delta(\chi)$.
We write $\sigma\wedge\tau=\delta$ if and only if there exists a linear cell $\delta$ such that $\overline{\sigma}\cap\overline{\tau}=\overline{\delta}$ holds.
\end{dfn}

\begin{rmk}
$\delta$ is allowed to be $\emptyset$.
\end{rmk}

\begin{dfn}
For a subset $K$ in $S^{m-1}$ we define the subset $M*K$ in $\widetilde{X}$ by 
\begin{equation}
M*K=\{(x,ty)\in M\times D^m\,;\, x\in M, y\in K, 0<t<1\}.
\end{equation}
\end{dfn}
\subsection{Intuitive representation}
Let us define a framework ($\mathcal{T},\mathcal{U},\mathcal{W}$) in which intuitive representation of local cohomology groups can be realized.
Let $\mathscr{F}$ be a sheaf of $\mathbb{C}$-vector spaces on $X$ and $\mathcal{T}$ a family of stratifications of $S^{m-1}$.
Firstly we assume the following conditions below for $\mathcal{T}$.

\begin{enumerate}
\renewcommand{\labelenumi}{(T-\arabic{enumi})}
\item Each $\chi\in\mathcal{T}$ is a stratification of $S^{m-1}$ associated with partitioning by a finite family of hyperplanes in $\mathbb{R}^m$ passing through the origin.
\item For any $\chi'$, $\chi''\in\tf$, there exists $\chi\in\tf$ such that $\chi'\prec\chi$ and $\chi''\prec\chi$.
\end{enumerate}
Secondly let $\mathcal{U}$ be a family of open neighborhoods of $M$ satisfying the following conditions.
\begin{enumerate}
\renewcommand{\labelenumi}{(U-\arabic{enumi})}
\item For any $\chi\in\mathcal{T}$, $\sigma\in\Delta(\chi)$ and for any $U\in\mathcal{U}$, we have
\begin{equation}
{\rm H}^k(M*{\rm St}_\chi(|\sigma|) \cap U,\mathscr{F})=0\quad(k\neq 0).
\end{equation}
\item For any $U_1\in\mathcal{U}$ and $U_2\in\mathcal{U}$, there exists $U\in\mathcal{U}$ with $U\subset U_1$ and $U\subset U_2$.
\end{enumerate}
Finally we denote by $\mathcal{W}$ a family of open sets in $X$ and assume the following conditions.
\begin{enumerate}
\renewcommand{\labelenumi}{(W-\arabic{enumi})}
\item For any $\chi\in\tf$, $\sigma\in\Delta(\chi)$ and $U\in\mathcal{U}$, we have $M*{{\rm St}_\chi(|\sigma|)}\cap U\in\mathcal{W}$.

\item (Existence of a finer asyclic stratification.) For any $W\in\mathcal{W}$, there exists $\chi\in\tf$, $\sigma\in\Delta(\chi)$ and $U\in\mathcal{U}$ such that $M*{{\rm St}_\chi(|\sigma|)}\cap U\subset W$.

\item (Cone connectivity of $W$.) Let $W\in\mathcal{W}$, $U\in\mathcal{U}$, $\sigma_1\in\Delta_{m-1}(\chi_1)$ and $\sigma_2\in\Delta_{m-1}(\chi_2)$ for some $\chi_1,\,\chi_2\in\tf$ satisfying $M*{\rm St}_{\chi_k}(|\sigma_k|) \cap U \subset W$  $(k=1,2)$.
Then there exist $\chi$ which is finer than $\chi_1$ and $\chi_2$, $(m-1)$-cells $\tau_1,\tau_2,\dots,\tau_\ell\in\Delta_{m-1}(\chi)$ and an open neighborhood $U'\in\mathcal{U}$ which satisfy the following conditions.
\begin{enumerate}
	\item $\tau_1\subset\sigma_1$ and $\tau_\ell\subset \sigma_2$,
	\item $M*{\rm St}(|\tau_k\wedge\tau_{k+1}|)\cap U'\subset W$ for $k=1,2,\dots,\ell-1$.
\end{enumerate}
\end{enumerate}
We call the above conditions the condition (T), (U) and (W), respectively.
We give some examples of $\mathcal{T}$, $\mathcal{U}$ and $\mathcal{W}$ satisfying the above conditions.
In the following examples, let $M$ be an open set in $\mathbb{R}^n$ and $X=M\times\sqrt{-1}\mathbb{R}^n$.
\begin{dfn}
Let $W$ be an open set in $\mathbb{C}^n$ and $\Gamma$ an $\mathbb{R}_+$-conic open subset in $\mathbb{R}^n$.
We say that $W$ is an infinitisimal wedge of type $M\times\sqrt{-1}\Gamma$ if and only if the following conditions hold.
\begin{enumerate}
\item $W \subset M \times \sqrt{-1}\Gamma$.
\item For any $\mathbb{R}_+$-conic open proper subset $\Gamma'$ of $\Gamma$, there exists an open neighborhood $U \subset X$ of $M$ such that
\begin{equation}
M \times \sqrt{-1}\Gamma' \cap U \subset W.
\end{equation}
\end{enumerate}
\end{dfn}
\begin{exm}
Let $\mathscr{F}$ be the sheaf $\mathscr{O}_X$ of holomorphic functions on $X$.
We set $\mathcal{T}$, $\mathcal{U}$ and $\mathcal{W}$ as follows.
\begin{flalign*}
\mathcal{T}&=\{\chi\,;\,\mbox{$\chi$ is associated with partitioning by the finite family of} \\
&\mbox{\qquad\qquad\qquad\qquad\qquad hyperplanes in $\mathbb{R}^n$ passing through the origin.}\}, \\
\mathcal{U}&=\{\mbox{all the Stein open neighborhoods of $M$ in $X$} \}, \\
\mathcal{W}&=\{\mbox{all the infinitisimal wedges of type $M\times\sqrt{-1}\Gamma$}\,; \\
&\mbox{\qquad\qquad\qquad\qquad$\Gamma$ runs through all the connected open cones.}\}.
\end{flalign*}
\end{exm}
\begin{exm}
Let $\mathscr{F}=\mathscr{O}_X$.
We set $\mathcal{T}=\{\chi\}$, where $\chi$ is the stratification which is associated by the set of hyperplanes $H_i=\{y=(y_1,y_2,\cdots,y_n)\in\mathbb{R}^n\,;\,y_i=0\}\quad(i=1,2,\cdots,n)$, and set $\mathcal{U}=\{X\}$.
Moreover we define $H^+_i$ and $H^-_i$ by
\begin{equation*}
H^\pm_i=\{y=(y_1,y_2,\cdots,y_n)\,;\,\pm y_i>0\},
\end{equation*}
respectively.
Then we set
\begin{equation}
\mathcal{W}=\{W\,;\,\mbox{Each $W$ is either $X$ itself or the finite intersection of $H^*_i$ with $*=+$ or $-$.}\}.
\end{equation}
\end{exm}

Now we are ready to define intuitive representation of local cohomology groups.
\begin{dfn}
We define intuitive representation $\hat{\rm H}(\mathscr{F}(\mathcal{W}))$ of local cohomology groups as follows.
\begin{eqnarray}
\hat{\rm H}(\mathscr{F}(\mathcal{W}))=\left(\pl\right)/\mathcal{R}.
\end{eqnarray}i
Here $\mathcal{R}$ is the $\mathbb{C}$-vector space generated by the following elements.
\begin{eqnarray}
f{\oplus}(-f|_{W_2})
\quad(f\in\mathscr{F}(W_1),\,W_1,\,W_2\in\mathcal{W}\ {\it with}\ W_2\subset W_1).
\end{eqnarray}
\end{dfn}

\begin{rmk}
We can consider a restriction map of intuitive representation.
Let $X$ be a topological manifold, $M$ a closed submanifold of $X$ and $\mathscr{F}$ a sheaf on $X$.
We assume that there exists $(\mathcal{T},\mathcal{U},\mathcal{W})$ satisfying the conditions (T), (U) and (W).
Moreover let $M'$ be an open subset in $M$ and a homeomorphism $\iota':X'\overset{\sim}{\longrightarrow}M'\times D^m$ with the following commutative diagram.
\begin{eqnarray}
\begin{CD}
X' @>\sim>\iota'> M'\times D^m \\
@VVV @VVV \\
X @>\sim>\iota> M\times D^m.
\end{CD}
\end{eqnarray}
We set 
\begin{flalign}
\mathcal{T}'&=\mathcal{T}, \\
\mathcal{U}'&=\mathcal{U}\cap X'=\{U\cap X'\,;\,U\in\mathcal{U}\}, \\
\mathcal{W}'&=\mathcal{W}\cap X'=\{W\cap X'\,;\,W\in\mathcal{W}\}.
\end{flalign}
Then $(\mathcal{T}',\mathcal{U}',\mathcal{W}')$ satisfy the conditions (T) and (W).
Hence by assuming the condition (U-1) for $\mathcal{T}'$ and $\mathcal{U}'$,
we can obtain intuitive representation of 
$$
{{\rm H}^m_{M'}(X',\mathscr{F})\underset{\mathbb{Z}_{M'}(M')}{\otimes}or_{M'/X'}(M')}
$$
by using $(\mathcal{T}',\mathcal{U}',\mathcal{W}')$.
Furthermore the restriction map of intuitive representation is induced from the one $\mathscr{F}(W)\longrightarrow\mathscr{F}(W\cap X')$ of the sheaf $\mathscr{F}$.
\end{rmk}


\section{The equivalence of local cohomology groups and their intuitive representation}

Through this section, we shall prove the following main theorem.
This theorem guarantees the equivalence of local cohomology groups and their intuitive representation.

\begin{thm}{\label{mainthm}}
The boundary value morphism $b_{\mathcal{W}}$ defined below gives an isomorphism from intuitive representation to local cohomology groups.
\begin{equation}
b_\mathcal{W}:\hat{\rm H}(\mathscr{F}(\mathcal{W}))\overset{\sim}{\longrightarrow}\hmx.
\end{equation}
\end{thm}

\subsection{Construction of the boundary value morphism $b_\mathcal{W}$}

First of all, we construct the boundary value morphism $b_\mathcal{W}$ from $\hat{\rm H}(\mathscr{F}(\mathcal{W}))$ to local cohomology groups thanks to Schapira's idea (p.497\cite{KS}).

Let $\mathscr{G}$ be a sheaf on $X$. We define the dual complex $D(\mathscr{G})$ of $\mathscr{G}$ by
\begin{equation*}
D(\mathscr{G})={\rm R}\mathscr{H}om_{\mathbb{C}_X}(\mathscr{G},\mathbb{C}_X).
\end{equation*}

Fix $W\in\mathcal{W}$.
By the condition (W-2), we can find $U\in\mathcal{U}$, $\chi\in\tf$ and $\sigma\in\Delta(\chi)$ satisfying $(M*{\rm St}(|\sigma|))\cap U\subset W$.
Hence we obtain the restriction map
\begin{equation}\label{rest}
\mathscr{F}(W)\rightarrow\mathscr{F}((M*{\rm St}(|\sigma|)) \cap U).
\end{equation}
We also get the restriction map of sheaves as $\overline{M*{\rm St}(|\sigma|)}\supset M$:
\begin{equation*}
\mathbb{C}_{\overline{M*{\rm St}(|\sigma|)}}\longrightarrow\mathbb{C}_M.
\end{equation*}
By applying $D(\bullet)$ and $\underset{\mathbb{C}_X}{\otimes}\mathbb{C}_U$ to the above morphism, we have
\begin{equation*}
\mathbb{C}_{(M*{\rm St}(|\sigma|)) \cap U}\longleftarrow D(\mathbb{C}_M).
\end{equation*}
\begin{rmk}\label{rmk1}
We note the following well-known facts.
\begin{equation}
D(\mathbb{C}_M)\simeq\mathbb{C}_M\underset{\mathbb{Z}_M}{\otimes} or_{M/X}[-m],
\end{equation}
\begin{equation}
{\rm H}^m_M(X,\mathbb{Z}_X)\simeq\Gamma(X, or_{M/X}),
\end{equation}
where $ or_{M/X}=\mathscr{H}^m_M(\mathbb{Z}_X)$ is the relative orientation sheaf on $M$.
\end{rmk}
Thus by applying ${\rm RHom}_{\mathbb{C}_X}(\bullet,\mathscr{F})$ to the above and taking the $0$-th cohomology, we obtain
\begin{equation}\label{bw}
\mathscr{F}((M*{\rm St}(|\sigma|)) \cap U)\rightarrow \hmx.
\end{equation}
We can get the morphism $b_W$ by composing \eqref{rest} and \eqref{bw};
\begin{equation}
\mathscr{F}(W)\rightarrow\mathscr{F}((M*{\rm St}(|\sigma|)) \cap U)\rightarrow \hmx.
\end{equation}
Then $b_W$ satisfies the proposition below.
\begin{prop}\label{indep}
The map $b_W$ does not depend on the choices of $U,\,\chi$ and $\sigma$.
\end{prop}

\begin{proof}
The fact that $b_W$ does not depend on the choices of $U\in\mathcal{U}$ is shown by the same argument as the following one.
Hence we only show that it is independent of the choices of $\chi$ and $\sigma$.
Let $\chi_1$ and $\chi_2$ be stratifications of $S^{m-1}$ which belong to $\mathcal{T}$, let $\sigma_1\in\Delta(\chi_1)$ and $\sigma_2\in\Delta(\chi_2)$.
We can find an open neighborhood $U'\in\mathcal{U}$, $\chi\in\mathcal{T}$ and $\tau_1,\dots,\tau_\ell\in\Delta_{m-1}(\chi)$ which satisfy the condition (W-3).
Furthermore, by (U-2), we may assume $U\supset U'$.
Then we obtain
\begin{equation*}
M*{\rm St}(|\tau_i|)\subset M*{\rm St}(|\tau_k\wedge\tau_{k+1}|)\quad(i=k,k+1).
\end{equation*}
Therefore we obtain the following commutative diagram for $i=k,k+1$.
\begin{equation*}
\xymatrix{
\mathbb{C}_{\overline{M*{\rm St}(|\tau_k\wedge\tau_{k+1}|)}} \ar[dr] \ar[r] &\mathbb{C}_{\overline{M*{\rm St}(|\tau_i|)}} \ar[d] \\
&\mathbb{C}_M.
}
\end{equation*}
By applying ${\rm RHom}(D(\bullet) \otimes \mathbb{C}_{U'},\mathscr{F})$ to the above diagram, we obtain the following commutative diagram for $i=k,k+1$.
\begin{equation*}
\xymatrix{
\mathscr{F}(W) \ar[r] \ar[dr] & \mathscr{F}((M*{\rm St}_\chi(|\tau_k\wedge\tau_{k+1}|)) \cap U') \ar[d] \ar[dr] & \\ 
 & \mathscr{F}((M*{\rm St}_\chi(|\tau_i|)) \cap U') \ar[r] & \hmx.
}
\end{equation*}
By the repeated applications of the same argument for $k=1,2,\dots,\ell-1$, we have
\begin{equation*}
\xymatrix{
\mathscr{F}(W) \ar[dr] \ar[r] & \mathscr{F}((M*{\rm St}_\chi(|\tau_1|)) \cap U') \ar[dr] & \\
 & \mathscr{F}((M*{\rm St}_\chi(|\tau_\ell|))\cap U') \ar[r] & \hmx.
}
\end{equation*}
Because of $\tau_1\subset\sigma_1$ and $\,\tau_\ell\subset\sigma_2$, we finally get the following commutative diagram.

\begin{equation*}
\xymatrix{
\mathscr{F}(W) \ar[r]  \ar[dr]  &\mathscr{F}((M*{\rm St}_\chi(|\sigma_1|)) \cap U') \ar[dr] & \\
  & \mathscr{F}((M*{\rm St}_\chi(|\sigma_2|)) \cap U') \ar[r] & \hmx.
}
\end{equation*}
Hence $b_W$ does not depend on the choices of $\chi$ and $\sigma$.
\end{proof}
The next corollary follows from Proposition \ref{indep} immediately.
\begin{cor}\label{cor1}
For any $W_1$ and $W_2$ in $\mathcal{W}$ with $W_2\subset W_1$, and for any $f\in\mathscr{F}(W_1)$, we have
\begin{equation}
b_{W_1}(f)=b_{W_2}(f|_{W_2}).
\end{equation}
\end{cor}
We finally extend $b_W$ to $b_\mathcal{W}$.
We get $b_\mathcal{W}$ by assigning $\underset{W\in\mathcal{W}}{\oplus}f_W\in\underset{W\in\mathcal{W}}{\oplus}\mathscr{F}(W)$ to $\displaystyle \sum_{W\in\mathcal{W}}b_W(f_W)$.
\begin{equation}
b_\mathcal{W}:\underset{W\in\mathcal{W}}{\oplus}\mathscr{F}(W)\longrightarrow\hmx.
\end{equation}
It follows from Corollary \ref{cor1} that $b_\mathcal{W}(f)=0$ holds for $f\in\mathcal{R}$.
Hence $b_\mathcal{W}$ passes through $\hat{\rm H}(\mathscr{F}(\mathcal{W}))$ and we obtain
\begin{equation}
b_\mathcal{W}:\hat{\rm H}(\mathscr{F}(\mathcal{W}))\longrightarrow\hmx\mathscr{a}.
\end{equation}

\subsection{The interpretation of the boundary value map $b_\mathcal{W}$}

Our next purpose is to compute the boundary value map $b_\mathcal{W}$ concretely.

Remember that all the cells are oriented.
Let $\mathds{1}$ denote a section which generates $or_{M/X}$ over $\mathbb{Z}_M$.
Each $\sigma\in\Delta_{m-1}(\chi)$ induces a section of $or_{M/X}(M)$, which is denoted by $\mathds{1}_\sigma$.

Let $\sigma\in\Delta(\chi)$ and $\tau\in\Delta(\chi')$ for $\chi,\chi'\in\mathcal{T}$.
We define $<\bullet,\bullet>$ which reflects the orientation of two cells under the following situations:
\begin{enumerate}
\item $\sigma$ and $\tau$ are cells of the same dimension and $\sigma\subset\tau$ or $\tau\subset\sigma$ holds.
\item $\chi=\chi'$, $\sigma\in\Delta_k(\chi)$ and $\tau\in\Delta_{k+1}(\chi)$ with $\sigma\subset\overline{\tau}$.
\end{enumerate}
\begin{dfn}
Let $\sigma$ and $\tau$ be cells satisfying the above condition either $1$ or $2$.
We define $<\sigma,\tau>$ by
\begin{eqnarray*}
<\sigma,\tau>=
\begin{cases}
1&(\text{the orientation of $\sigma$ or the induced orientation from $\sigma$ has}\\
&\text{the same orientation of $\tau$})\\
-1&(\text{otherwise}).
\end{cases}
\end{eqnarray*}
\end{dfn}

Furthermore, we introduce a mapping whose image is in locally constant functions on $M$
$$
a_\mathds{1}(\bullet):\Delta_{m-1}(\chi)\longrightarrow\mathbb{Z}_M(M)
$$
defined by the formula
$$
\mathds{1}_\sigma=a_\mathds{1}(\sigma)\cdot\mathds{1}\quad(\sigma\in\Delta_{m-1}(\chi)).
$$

We construct an exact sequence which is needed when we compute the concrete expression of $b_\mathcal{W}$.
Let $\chi$ be in $\mathcal{T}$.
\begin{dfn}
Let $k=0,1,\cdots,m-1$.
We define the sheaf $\mathscr{L}^{k-m+1}_\chi$ as follows.
\begin{eqnarray}
\mathscr{L}^{k-m+1}_\chi=\underset{\sigma\in\Delta_{k}(\chi)}{\oplus}\mathbb{C}_{\overline{\sigma}},
\end{eqnarray}
where $\mathbb{C}_{\overline{\sigma}}=\mathbb{C}_{\overline{M*{\rm St}(|\sigma|)}}$.
More precisely, $\mathbb{C}_{\overline{\sigma}}$ is the sheaf such that $\mathbb{C}_{\overline{\sigma}}(U)$ consists of pairs $(\sigma, s)$ with $s \in \mathbb{C}_{\overline{M * {\rm St}(|\sigma|)}}(U)$ for an open set $U\subset X$.

\end{dfn}

For convenience, we set $\mathscr{L}^{-m}_\chi=\mathbb{C}_X$.
We also note that, for $\sigma \subset \tau$, the image of $c_\sigma \in \mathbb{C}_{\overline{\sigma}}$ by the canonical map 
$\mathbb{C}_{\overline{\sigma}} \to \mathbb{C}_{\overline{\tau}}$
is denoted by $c_\sigma|_\tau$ for simplicity.

Then we have the following sequence.
\begin{eqnarray}{\label{mainexact}}
0\longrightarrow
\mathscr{L}^{-m}_\chi\underset{\mathbb{Z}_X}{\otimes} i_*or_{M/X}\overset{d^{-m}_\chi}{\longrightarrow}
\cdots\overset{d^{-1}_\chi}{\longrightarrow}
\mathscr{L}^{0}_\chi\underset{\mathbb{Z}_X}{\otimes} i_*or_{M/X}\overset{d^{0}_\chi}{\longrightarrow}
\mathbb{C}_M\longrightarrow0,
\end{eqnarray}
where the map $i$ is the embedding $i:M\longrightarrow X$.
Note that, in the above sequence, the leftmost term $\mathscr{L}^{-m}_\chi\underset{\mathbb{Z}_X}{\otimes} i_*or_{M/X}$ is located at degree $-m$ and the term $\mathscr{L}^{0}_\chi\underset{\mathbb{Z}_X}{\otimes} i_*or_{M/X}$ is at degree $0$.
Here, for any $\sigma\in\Delta_{k}(\chi)\ (k=0,1,\cdots,m-2)$ and for any $c_\sigma\in \mathbb{C}_{\overline{\sigma}}$, we define $d^{k-m+1}_\chi$ as follows.
\begin{eqnarray*}
d^{k-m+1}_\chi(c_\sigma\otimes \ell\mathds{1})=\underset{\substack{\sigma\subset\overline{\tau} \\ \tau\in\Delta_{k+1}(\chi)}}{\oplus} (<\sigma,\tau> c_\sigma|_\tau \otimes \ell\mathds{1}),
\end{eqnarray*}
where $\ell\mathds{1}\in i_*or_{M/X}$ with $\ell\in\mathbb{Z}_M(M)$.
We also define $d^{-m}_\chi$ and $d^0_\chi$ by
\begin{eqnarray*}
d^{-m}_\chi(c\otimes \ell\mathds{1})&=&\underset{\sigma\in\Delta_0(\chi)}{\oplus}c|_\sigma\otimes \ell\mathds{1}, \\
d^0_\chi(c_\sigma\otimes \ell\mathds{1})&=&\ell a_\mathds{1}(\sigma)c_\sigma,
\end{eqnarray*}
where $c\in\mathbb{C}_X$, $c_\sigma\in \mathbb{C}_{\overline{\sigma}}$ with $\sigma\in\Delta_{m-1}(\chi)$ and $\ell\mathds{1}\in i_*or_{M/X}$ with $\ell\in\mathbb{Z}_M(M)$.
We remark that $d^0_\chi$ does not depend on the choices of $\mathds{1}$.

We write $d^k$ and $\mathscr{L}^k_x$ instead of $d^k_\chi$ and $(\mathscr{L}^k_\chi)_x$, respectively if there is no risk of confusion.

\begin{lem}
We have $d^{k+1}\circ d^{k}=0$ for $k=-m,-m+1,\cdots,-1$.
\end{lem}
\begin{proof}
Firstly we prove $d^{0}\circ d^{-1}=0$.
Let $c_\sigma\in\mathbb{C}_{\overline{\sigma}}$ with $\sigma\in\Delta_{m-2}(\chi)$ and $\mathds{1}\in i_*or_{M/X}$.
Here we remark that, there exists just two cells whose closure contains $\sigma$.
Let $\tau$ and $\tau'$ be mutually distinct cells whose closure contains $\sigma$.
Then we obtain
\begin{flalign*}
&d^{0}\circ d^{-1}(c_\sigma\otimes\mathds{1}) \\
=\,&d^0(<\sigma,\tau> c_\sigma|_\tau \otimes \mathds{1}\,\oplus<\sigma,\tau'> c_\sigma|_{\tau'} \otimes \mathds{1} )\\
=\,&<\sigma,\tau> a_\mathds{1}(\tau)c_\sigma|_\tau + <\sigma,\tau'> a_\mathds{1}(\tau')c_\sigma|_{\tau'} \\
=\,&0.
\end{flalign*}
Here the last equation follows from the fact
$$
<\sigma,\tau>a_\mathds{1}(\tau)=-<\sigma,\tau'>a_\mathds{1}(\tau').
$$

Next we prove $d^{k+1}\circ d^{k}=0$ for $k=-m,-m+1,\cdots,-2$.
We remark that for any $\sigma\in\Delta_k(\chi)$ and $\tau\in\Delta_{k+1}(\chi)$ we get $<\sigma,\tau>$ by comparing the orientation induced by $\tau$ with the orientation of $\sigma$.
Here the orientation of $\sigma$ is determined so that the outward pointing vector of $\tau$ followed by a positive  frame of $\sigma$ from that of $\tau$.
Note that, for any $\delta\in\Delta_{k+2}(\chi)$ and $\sigma\in\Delta_k(\chi)$ which satisfy $\sigma\subset\overline{\delta}$,
there exist just two $(k+1)$-cells which are contained in the closure of $\delta$ and whose closure contain $\sigma$.
We denote by them $\tau$ and $\tau'$.
By the above observation we have
\begin{flalign*}
d^{k+1}\circ d^{k}(c_\sigma)
=\,&d^{k+1}(\underset{\substack{\sigma\subset\overline{\tau} \\ \tau\in\Delta_{k+1}(\chi)}}{\oplus} (<\sigma,\tau> c_\sigma|_\tau \otimes \mathds{1})) \\
=\,&\underset{\substack{\tau\subset\overline{\delta} \\ \delta\in\Delta_{k+2}(\chi)}}{\oplus}
\underset{\substack{\sigma\subset\overline{\tau} \\ \tau\in\Delta_{k+1}(\chi)}}{\oplus} (<\sigma,\tau><\tau,\delta> c_\sigma|_\delta \otimes \mathds{1}) \\
=\,&\underset{\substack{\sigma\subset\overline{\delta} \\ \delta\in\Delta_{k+2}(\chi)}}{\oplus}
(<\sigma,\tau><\tau,\delta>\ +<\sigma,\tau'><\tau',\delta>)c_\sigma|_\delta \otimes \mathds{1} \\
=\,&0,
\end{flalign*}
where $c_\sigma\in\mathbb{C}_{\overline{\sigma}}$ and $\mathds{1}\in i_*or_{M/X}$.
\end{proof}
By this lemma we see that the sequence \eqref{mainexact} is a complex.
\begin{prop}\label{exact}
The complex \eqref{mainexact} is exact.
\end{prop}
Before starting the proof of Proposition \ref{exact}, we introduce some definitions and lemmas.
\begin{dfn}
The set $\widetilde{\Delta}_k(\chi,x)$ is defined as follows.
If $x\in M$, we define $\widetilde{\Delta}_k(\chi,x)=\Delta_k(\chi)$,
and if $x\notin M$, we define
\begin{equation}
\widetilde{\Delta}_k(\chi,x)=\{\tau\in\Delta_k(\chi)\mid\tau\subset\overline{{\rm St}(|\sigma_x|)}\},
\end{equation}
where $\sigma_x\in\Delta(\chi)$ is the unique cell with $x\in M*\sigma_x$.
\end{dfn}
Then the following lemma holds.
\begin{lem}{\label{isoL}}
We have the isomorphism $\alpha^k$
\begin{equation}
\alpha^k:{\rm Hom}_\mathbb{C}(\widetilde{\Delta}_k(\chi,x),\mathbb{C})\overset{\sim}{\longrightarrow}\mathscr{L}^{k-m+1}_x.
\end{equation}
\end{lem}

The proof follows immediately from correspondence below.
\begin{eqnarray}\label{corres}
\begin{array}{ccc}
\alpha^k:{\rm Hom}_\mathbb{C}(\widetilde{\Delta}_k(\chi,x),\mathbb{C})   & \longrightarrow & \mathscr{L}^{k-m+1}_x \\[2pt]
\rotatebox{90}{$\in$} &                 & \rotatebox{90}{$\in$} \\[-1pt]
\varphi & \longmapsto & \underset{\sigma\in\widetilde{\Delta}_k(\chi,x)}{\oplus}\varphi(\sigma)1_\sigma.
\end{array}
\end{eqnarray}
Here $1_\sigma$ denotes the image of $1_X$ by $\mathbb{C}_X(X)\longrightarrow\mathbb{C}_{\overline{\sigma}}(X)$.
\begin{lem}{\label{lem2}}
Two sequences
\begin{equation}{\label{delta}}
0\longleftarrow\widetilde{\Delta}_0(\chi,x)\overset{\partial}{\longleftarrow}\dots\overset{\partial}{\longleftarrow}\widetilde{\Delta}_{m-1}(\chi,x)\longleftarrow0,
\end{equation}
\begin{equation}{\label{Lseq}}
0{\longrightarrow}\mathscr{L}^{-m+1}_x\overset{d^{-m+1}_\chi}{\longrightarrow}
\cdots\overset{d^{-1}_\chi}{\longrightarrow}
\mathscr{L}^{0}_x\longrightarrow0
\end{equation}
are dual to each other. Here $\partial$ is the boundary operator of the simplicial complex.
\end{lem}
\begin{proof}
By applying the exact functor ${\rm Hom}_\mathbb{C}(\bullet,\mathbb{C})$ to \eqref{delta}, we obtain
\begin{equation}{\label{dual}}
0\longrightarrow{\rm Hom}_\mathbb{C}(\widetilde{\Delta}_0(\chi,x),\mathbb{C})\overset{\partial^\ast}{\longrightarrow}\dots\overset{\partial^\ast}{\longrightarrow}{\rm Hom}_\mathbb{C}(\widetilde{\Delta}_{m-1}(\chi,x),\mathbb{C})\longrightarrow0,
\end{equation}
where $\partial^*={\rm Hom}_\mathbb{C}(\partial,\mathbb{C})$.
Due to Lemma \ref{isoL} we just prove the commutativity of the diagram below.
\begin{eqnarray}{\label{dia}}
\begin{CD}
\mathscr{L}^{k-m+1}_x @>d^{k-m+1}>> \mathscr{L}^{k-m+2}_x\\
@A\alpha^k AA @AA\alpha^{k+1} A \\
{\rm Hom}_\mathbb{C}(\widetilde{\Delta}_k(\chi,x),\mathbb{C}) @>>\partial^\ast> {\rm Hom}_\mathbb{C}(\widetilde{\Delta}_{k+1}(\chi,x),\mathbb{C}).
\end{CD}
\end{eqnarray}
Here we recall that 
\begin{equation*}
\partial^\ast(\varphi)=\underset{\sigma\subset\overline{\tau}}{\oplus}<\sigma,\tau>\varphi(\sigma).
\end{equation*}
Then we have
$$
(d^{k-m+1}\circ \alpha^k)(\varphi)
=\underset{\substack{\tau\in\Delta_{k+1}(\chi,x) \\ \sigma\subset\overline{\tau}}}{\oplus} <\sigma,\tau> \varphi(\sigma)1_\tau.
$$
On the other hand, we have
$$
(\alpha^{k+1}\circ \partial^*)(\varphi)
=\underset{\substack{\sigma\in\Delta_{k}(\chi,x) \\ \sigma\subset\overline{\tau}}}{\oplus} <\sigma,\tau> \varphi(\sigma)1_\tau.
$$
Note that both the indices sets in the above two computations coincide because
\begin{flalign*}
&\{ (\sigma,\tau)\, |\, \sigma\subset\overline{\tau}, \sigma\in\Delta_{k}(\chi,x) \}
=\{ (\sigma,\tau)\, |\, \sigma\subset\overline{\tau}, \sigma\in\Delta_{k}(\chi),\, \sigma\subset\overline{{\rm St}(|\delta_x|)} \} \\
=&\{ (\sigma,\tau)\, |\, \sigma\subset\overline{\tau}, \sigma\subset\overline{\tau}\subset\overline{{\rm St}(|\delta_x|)} \}
=\{ (\sigma,\tau)\, |\, \sigma\subset\overline{\tau}, \tau\in\Delta_{k+1}(\chi),\, \tau\subset\overline{{\rm St}(|\delta_x|)} \} \\
=&\{ (\sigma,\tau)\, |\, \sigma\subset\overline{\tau}, \tau\in\Delta_{k+1}(\chi,x) \},
\end{flalign*}
where the cell $\delta_x$ is the unique cell satisfying $x\in M*\delta_x$.
Hence \eqref{dia} becomes a commutative diagram.
\end{proof}
Now we are ready to prove Proposition \ref{exact}.
\begin{proof}
It is enough to prove that the following sequence is exact.
\begin{eqnarray*}
0\longrightarrow
\mathscr{L}^{-m}_x\underset{\mathbb{Z}}{\otimes}(i_*or_{M/X})_x\overset{d^{-m}}{\longrightarrow}
\cdots\overset{d^{-1}}{\longrightarrow}
\mathscr{L}^{0}_x\underset{\mathbb{Z}}{\otimes}(i_*or_{M/X})_x\overset{d^{0}}{\longrightarrow}
(\mathbb{C}_M)_x\longrightarrow0.
\end{eqnarray*}
By Lemma \ref{lem2}, the claim follows from the homology groups of $\overline{{\rm St}(|\delta_x|)}$ if $x \notin M$ and $S^{m-1}$ if $x \in M$.
\begin{enumerate}
\item  The case of $x\notin M$.

Since $\overline{{\rm St}(|\delta_x|)}$ is contractible to a point in $|\delta_x|$, its homology groups vanish except for the $0$-th group, that is
\begin{equation*}
{\rm H}_k(\overline{{\rm St}(|\delta_x|)},\mathbb{C})=0\quad(k\neq 0).
\end{equation*}
Hence the complex \eqref{Lseq} is concentrated in degree $-m+1$ and its $(-m+1)$-th cohomology group is $\mathbb{C}$, and we have the exact sequence:
\begin{eqnarray*}
0\longrightarrow
\mathscr{L}^{-m}_x\underset{\mathbb{Z}}{\otimes}(i_*or_{M/X})_x\overset{d^{-m}}{\longrightarrow}
\cdots\overset{d^{-1}}{\longrightarrow}
\mathscr{L}^{0}_x\underset{\mathbb{Z}}{\otimes}(i_*or_{M/X})_x\longrightarrow0.
\end{eqnarray*}

\item The case of $x\in M$.

Except for the $0$-th and the $(m-1)$-th homology groups, the homology groups of the $(m-1)$-dimensional spherical surface vanish.
Its $0$-th homology group is isomorphic to $\mathbb{C}$,
and the $(m-1)$-th homology gives the orientation of the surface which is isomorphic to $\mathbb{C}\underset{\mathbb{Z}}{\otimes}(or_{M/X})_x$.
Hence we obtain the exact sequence;
\begin{eqnarray*}{\label{ex-sequence}}
0\longrightarrow
\mathscr{L}^{-m}_x\underset{\mathbb{Z}}{\otimes}(i_*or_{M/X})_x\overset{d^{-m}}{\longrightarrow}
\cdots\overset{d^{-1}}{\longrightarrow}
\mathscr{L}^{0}_x\underset{\mathbb{Z}}{\otimes}(i_*or_{M/X})_x\overset{d^{0}}{\longrightarrow}
(\mathbb{C}_M)_x\longrightarrow0.
\end{eqnarray*}

\end{enumerate}
The proof of Proposition \ref{exact} has been completed.
\end{proof}

By Proposition \ref{exact} we have the quasi-isomorphism.
\begin{equation}{\label{quasi}}
\mathscr{L}^\bullet_\chi\underset{\mathbb{Z}_X}{\otimes} i_*or_{M/X} \simeq\mathbb{C}_M.
\end{equation}
Here $\mathscr{L}^\bullet_\chi\underset{\mathbb{Z}_X}{\otimes} i_*or_{M/X}$ designates the complex
\begin{equation}
0\longrightarrow
\mathscr{L}^{-m}_\chi\underset{\mathbb{Z}_X}{\otimes} i_*or_{M/X} \longrightarrow
\mathscr{L}^{-m+1}_\chi\underset{\mathbb{Z}_X}{\otimes} i_*or_{M/X} \longrightarrow
\cdots\longrightarrow
\mathscr{L}^{0}_\chi\underset{\mathbb{Z}_X}{\otimes} i_*or_{M/X} \longrightarrow0.
\end{equation}

Let $U$ be in $\mathcal{U}$.
For convenience we set 
\begin{eqnarray}
\mathscr{F}(\sigma,U)&=&\mathscr{F}(M*{\rm St}(|\sigma|)\cap U)\quad(\sigma\in\Delta(\chi)), \\
\mathscr{F}_k(\chi,U)&=&\underset{\sigma\in\Delta_{k-1}(\chi)}{\oplus}\mathscr{F}(\sigma,U) \quad(k=1,2,\cdots,m).
\end{eqnarray}
Moreover we set
\begin{eqnarray}
\mathscr{F}(\sigma)&=&\ilim[U\in\mathcal{U}]\mathscr{F}(\sigma,U), \\
\mathscr{F}_k(\chi)&=&\ilim[U\in\mathcal{U}]\mathscr{F}_k(\chi,U) \quad(k=1,2,\cdots,m).
\end{eqnarray}
\begin{lem}\label{chi}
We have the isomorphism $\beta_\chi$ as follows.
\begin{equation}\label{lcg}
\begin{split}
\frac{\mathscr{F}_m(\chi)}
{{\rm Im}(\mathscr{F}_{m-1}(\chi)\rightarrow\mathscr{F}_m(\chi))}
\underset{\mathbb{Z}_M(M)}{\otimes}or_{M/X}(M)
{\overset{\sim}{\longrightarrow}}\hmx.
\end{split}
\end{equation}
\end{lem}
\begin{proof}
By applying $D(\bullet)$ to \eqref{quasi} we obtain
$$
D(\mathscr{L}^\bullet_\chi\underset{\mathbb{Z}_X}{\otimes} i_*or_{M/X})\overset{\sim}{\longleftarrow}D(\mathbb{C}_M).
$$
Because the support of $D(\mathbb{C}_M)$ is contained in $M$ we obtain, for $U\in\mathcal{U}$,
\begin{flalign}{\label{lm}}
D(\mathbb{C}_M)&=D(\mathbb{C}_M)\underset{\mathbb{Z}_X}{\otimes}\mathbb{C}_U. \nonumber \\
&\overset{\sim}{\longrightarrow}D(\mathscr{L}^\bullet_\chi\underset{\mathbb{Z}_X}{\otimes} i_*or_{M/X})\underset{\mathbb{Z}_X}{\otimes}\mathbb{C}_U.
\end{flalign}
Moreover, since each $M*{\rm St}(|\sigma|)$ is cohomologically trivial we get
$$
D(\mathscr{L}^\bullet_\chi\underset{\mathbb{Z}_X}{\otimes} i_*or_{M/X})\underset{\mathbb{Z}_X}{\otimes}\mathbb{C}_U\simeq \mathscr{N}^\bullet_\chi\underset{\mathbb{Z}_X}{\otimes}i_*or_{M/X}.
$$
Here the $\mathscr{N}^\bullet_\chi$ designates the complex
$$
0\longrightarrow
\mathscr{N}^{0}_\chi\longrightarrow
\mathscr{N}^{1}_\chi\longrightarrow
\cdots\longrightarrow
\mathscr{N}^{m-1}_\chi\longrightarrow0,
$$
where
$$
\mathscr{N}^{k}_\chi=\underset{\sigma\in\Delta_{m-k-1}(\chi)}{\oplus}\mathbb{C}_{{\sigma},U},
$$
$$
\mathbb{C}_{\sigma,U}=\mathbb{C}_{M*{\rm St}(|\sigma|)\cap U},
$$
for $k=0,1,\cdots,m-2$ and $\mathscr{N}^{m-1}_\chi=\mathbb{C}_U$.
By applying ${\rm RHom}_{\mathbb{C}_X}(\bullet,\mathscr{F})$ to \eqref{lm} we obtain
\begin{equation}{\label{homM}}
{\rm RHom}_{\mathbb{C}_X}(\mathscr{N}^\bullet_\chi\underset{\mathbb{Z}_X}{\otimes}i_*or_{M/X},\mathscr{F})\overset{\sim}{\longrightarrow}
{\rm RHom}_{\mathbb{C}_X}(D(\mathbb{C}_M),\mathscr{F}).
\end{equation}
Since $U$ satisfies the condition (U-1), we obtain
$$
{\rm RHom}_{\mathbb{C}_X}(\mathscr{N}^\bullet_\chi\underset{\mathbb{Z}_X}{\otimes}i_*or_{M/X},\mathscr{F})=
{\rm Hom}_{\mathbb{C}_X}(\mathscr{N}^\bullet_\chi\underset{\mathbb{Z}_X}{\otimes}i_*or_{M/X},\mathscr{F}).
$$
By taking the $0$-th cohomology, we get
$$
\frac{\mathscr{F}_m(\chi,U)}
{{\rm Im}(\mathscr{F}_{m-1}(\chi,U)\rightarrow\mathscr{F}_m(\chi,U))}
\underset{\mathbb{Z}_M(M)}{\otimes}or_{M/X}(M)
\overset{\sim}{\longrightarrow} \hmx.
$$
Finally we take the inductive limit in the above isomorphism with respect to $U\in\mathcal{U}$, and then, the claim of the lemma follows.
\end{proof}
Hereafter we identify the objects in both sides of \eqref{lcg} by $\beta_\chi$.
For any $\sigma\in\Delta_{m-1}(\chi)$ the following diagram commutes.
\begin{eqnarray}\label{commuteL}
\begin{CD}
\mathscr{L}^{0}_\chi\underset{\mathbb{Z}_X}{\otimes} i_*or_{M/X} @>>> \mathbb{C}_M\\
@A\varepsilon_{\chi,\sigma} AA @AAA \\
\mathbb{C}_{\overline{\sigma}}\underset{\mathbb{Z}_X}{\otimes} i_*or_{M/X} @>\sim>\eta_{\chi,\sigma}> \mathbb{C}_{\overline{M*{\rm St}(|\sigma|)}}.
\end{CD}
\end{eqnarray}
Here $\varepsilon_{\chi,\sigma}$ is the embedding map and $\eta_{\chi,\sigma}$ is defined by
\begin{eqnarray}
c_\sigma{\otimes}\,\mathds{1}\mapsto a_\mathds{1}(\sigma)\cdot c_\sigma.
\end{eqnarray}
Note that $\eta^{-1}_{\chi,\sigma}$ is given by
\begin{equation}{\label{eta}}
c \mapsto a_\mathds{1}(\sigma)c \otimes\mathds{1},
\end{equation}
where $c\in\mathbb{C}_{\overline{M*{\rm St}(|\sigma|)}}$.
By \eqref{commuteL} the diagram of complexes below commutes.
\begin{eqnarray*}
\begin{CD}
\mathscr{L}^\bullet_\chi\underset{\mathbb{Z}_X}{\otimes} i_*or_{M/X} @>\sim>> \mathbb{C}_M\\
@A\varepsilon_{\chi,\sigma} AA @AAA \\
\mathbb{C}_{\overline{\sigma}}\underset{\mathbb{Z}_X}{\otimes} i_*or_{M/X} @>\sim>\eta_{\chi,\sigma}> \mathbb{C}_{\overline{M*{\rm St}(|\sigma|)}}.
\end{CD}
\end{eqnarray*}
By applying ${\rm RHom}_{\mathbb{C}_X}(D(\bullet)\otimes\mathbb{C}_U,\mathscr{F})$ to the above commutative diagram, taking the $0$-th cohomology groups and taking the inductive limit $\ilim[U\in\mathcal{U}]$ of these cohomology groups, we obtain the commutative diagram.
\begin{eqnarray}\label{comdia3}
\begin{CD}
\dfrac{\mathscr{F}_m(\chi)}
{{\rm Im}(\mathscr{F}_{m-1}(\chi)\rightarrow\mathscr{F}_m(\chi))}
\underset{\mathbb{Z}_M(M)}{\otimes}or_{M/X}(M) @>\sim>\beta_{\chi}> \hmx \\
@A\varepsilon^0_{\chi,\sigma} AA @AAb_{W_{|\sigma|}} A \\
\mathscr{F}(\sigma)\underset{\mathbb{Z}_M(M)}{\otimes}or_{M/X}(M) @>\sim>\eta^0_{\chi,\sigma}> \ilim[U\in\mathcal{U}]\mathscr{F}(M*{\rm St}(|\sigma|)\cap U).
\end{CD}
\end{eqnarray}
Here vertical map on the right side is $b_{W_{|\sigma|}}=\ilim[U\in\mathcal{U}] b_{W_{|\sigma|,U}}$ with $W_{|\sigma|,U}=M*{\rm St}_\chi(|\sigma|) \cap U$ and
$$
\varepsilon^0_{\chi,\sigma}=\ilim[U\in\mathcal{U}]{\rm H}^0({\rm RHom}_{\mathbb{C}_X}(D(\varepsilon_{\chi,\sigma})\otimes\mathbb{C}_U,\mathscr{F})).
$$
Moreover the upper horizontal map is $\beta_\chi$ and
$$
\eta^0_{\chi,\sigma}=\ilim[U\in\mathcal{U}]{\rm H}^0({\rm RHom}_{\mathbb{C}_X}(D(\eta_{\chi,\sigma})\otimes\mathbb{C}_U,\mathscr{F})).
$$
By taking a direct sum with respect to $\sigma\in\Delta_{m-1}(\chi)$, we obtain
\begin{eqnarray}{\label{main com}}
\begin{CD}
\dfrac{\mathscr{F}_m(\chi)}
{{\rm Im}(\mathscr{F}_{m-1}(\chi)\rightarrow\mathscr{F}_m(\chi))}
\underset{\mathbb{Z}_M(M)}{\otimes}or_{M/X}(M) @>\sim>\beta_{\chi}> \hmx \\
@A\varepsilon^0_{\chi} AA @AA\underset{|\sigma|\in|\Delta|_{m-1}(\chi)}{\oplus}b_{W_{|\sigma|}} A \\
\underset{\sigma\in\Delta_{m-1}(\chi)}{\oplus}\left(\mathscr{F}(\sigma)\underset{\mathbb{Z}_M(M)}{\otimes}or_{M/X}(M)\right) @>\sim>\eta^0_{\chi}> \underset{|\sigma|\in|\Delta|_{m-1}(\chi)}{\oplus}\left(\ilim[U\in\mathcal{U}]\mathscr{F}(M*{\rm St}(|\sigma|)\cap U)\right),
\end{CD}
\end{eqnarray}
where $\eta^0_{\chi}=\underset{\sigma\in\Delta_{m-1}(\chi)}{\oplus}\eta^0_{\chi,\sigma}$ and $\varepsilon^0_{\chi}=\underset{\sigma\in\Delta_{m-1}(\chi)}{\oplus}\varepsilon^0_{\chi,\sigma}$.
We note that $\varepsilon^0_{\chi}$ is surjective.

On account of \eqref{eta},
tracing the map $(\eta^0_\chi)^{-1}$ and $\varepsilon^0_{\chi}$ in \eqref{comdia3} we get the proposition below.
\begin{prop}\label{main1}
The boundary value morphism $b_\mathcal{W}$
\begin{align}
\begin{aligned}
\hat{\rm H}(\mathscr{F}(\mathcal{W}))
\longrightarrow
\hmx
\end{aligned}
\end{align}
is given by, for $W \in \mathcal{W}$ and $f \in \mathscr{F}(W)$
\begin{equation}
f\mapsto f a_\mathds{1}(\sigma) \otimes\mathds{1}
\in 
\mathscr{F}(\sigma,U)
\otimes or_{M/X}(M)
\subset
\mathscr{F}_m(\chi) \otimes or_{M/X}(M),
\end{equation}
through the identification by $\beta_\chi$.
Here $\mathds{1}$ is a section of $or_{M/X}(M)$ which generates $or_{M/X}$ over $\mathbb{Z}_M$, $\sigma \in \Delta_{m-1}(\chi)$ and 
$U \in \mathcal{U}$ are taken such that $M*{\rm St}_\chi(|\sigma|) \cap U \subset W$ holds.

\end{prop}
Note that $a_{\mathds{1}}(\sigma)\otimes\mathds{1}$ does not depend on the choices of the generator $\mathds{1}$ because another choice of the generator is, if $M$ is connected, $-\mathds{1}$, and hence, we get on each connected component
$$
a_{-\mathds{1}}(\sigma)\otimes-\mathds{1}=a_{\mathds{1}}(\sigma)\otimes\mathds{1}.
$$
\subsection{The inverse map $\rho$ of $b_\mathcal{W}$}

The aim of this subsection is to construct the inverse map $\rho$ of $b_\mathcal{W}$ and complete the proof for Theorem \ref{mainthm}.

Firstly we prove the commutativity of diagram [1],[2] and [3].
\begin{equation}\label{fivecom}
\xymatrix@C=0pt{
\dfrac{\mathscr{F}_m(\chi)}
{{\rm Im}(\mathscr{F}_{m-1}(\chi)\rightarrow\mathscr{F}_m(\chi))}
\underset{\mathbb{Z}_M(M)}{\otimes}or_{M/X}(M) \ar[rr]^{\beta_\chi} \ar[dr]^{\rho_\chi} & & \hmx \\
& \hat{\rm H}(\mathcal{W}(\mathscr{F})) \ar[ur]^{b_\mathcal{W}} \ar@{}[ld]|{[2]} \ar@{}[u]|{[3]} \ar@{}[r]|{[1]} & \\
\underset{\sigma\in\Delta_{m-1}(\chi)}{\oplus}\left(\mathscr{F}(\sigma)\underset{\mathbb{Z}_M(M)}{\otimes}or_{M/X}(M)\right) \ar@{->>}[uu]^{\varepsilon^0_\chi} \ar[rr]_{\eta^0_\chi} & & \underset{|\sigma|\in|\Delta|_{m-1}(\chi)}{\oplus}\left(\ilim[U\in\mathcal{U}]\mathscr{F}(M*{\rm St}(|\sigma|)\cap U)\right) \ar[ul]_\varpi \ar[uu]_{\underset{|\sigma|\in|\Delta|_{m-1}(\chi)}{\oplus}b_{W_{|\sigma|}}}.
}
\end{equation}
Here $\varpi$ is a canonical morphism
and $\rho_\chi$ is defined as follows.
\begin{dfn}{\label{rho}}
The map $\rho_\chi$
\begin{align}
\begin{aligned}
\rho_\chi:
\dfrac{\mathscr{F}_m(\chi)}
{{\rm Im}(\mathscr{F}_{m-1}(\chi)\rightarrow\mathscr{F}_m(\chi))}
\underset{\mathbb{Z}_M(M)}{\otimes}or_{M/X}(M)
\longrightarrow
\hat{\rm H}(\mathscr{F}(\mathcal{W}))
\end{aligned}
\end{align}
is defined by, for $f_\sigma \in \mathscr{F}(\sigma)$,
$$
f_\sigma\underset{\mathbb{Z}}{\otimes}\mathds{1} \mapsto a_\mathds{1}(\sigma)\cdot f_\sigma,
$$
where $\mathds{1}$ is a section of $or_{M/X}(M)$ which generates $or_{M/X}$ over $\mathbb{Z}_M$.
\end{dfn}
We note that $\rho_\chi$ is well-defined.
As a matter of fact, the set ${{\rm Im}(\mathscr{F}_{m-1}(\chi)\rightarrow\mathscr{F}_m(\chi))}$ is generated by elements
$$
<\tau,\sigma^+> f|_{\sigma^+} \underset{\mathbb{Z}}{\otimes}\mathds{1} \,+ <\tau,\sigma^-> f|_{\sigma^-} \underset{\mathbb{Z}}{\otimes}\mathds{1},
$$
for $f\in\mathscr{F}_{m-2}(\tau)$, $\tau\in\Delta_{m-2}(\chi)$ and $\sigma^*\in\Delta_{m-1}(\chi)$ with $\tau\subset\overline{\sigma^*}$.
Here $*=+$ or $-$ and $\sigma^+ \neq \sigma^-$.
Hence we obtain
\begin{flalign*}
&\rho_\chi(<\tau,\sigma^+> f|_{\sigma^+} \underset{\mathbb{Z}}{\otimes}\mathds{1} \,+ <\tau,\sigma^-> f|_{\sigma^-} \underset{\mathbb{Z}}{\otimes}\mathds{1}) \\
=&(a_\mathds{1}(\sigma^+) <\tau,\sigma^+> +\,a_\mathds{1}(\sigma^-) <\tau,\sigma^-> )f =0.
\end{flalign*}

The commutativity of [1] follows from Corollary \ref{cor1}.
And the diagram [2] also commutes by Definition \ref{rho} and the definition of $\varpi$.
We prove the commutativity of the diagram [3].
Since $\varepsilon^0_{\chi}$ is surjective, it suffice to prove that
$$
\beta_\chi \circ \varepsilon^0_{\chi}
=b_\mathcal{W} \circ \rho_\chi \circ \varepsilon^0_{\chi},
$$
and this commutativity follows from \eqref{main com} and the commutativity of [1] and [2].
Then we get the following
\begin{thm}{\label{rho1}}
The map
$$
\rho_\chi\circ\beta^{-1}_\chi:
\hmx
\longrightarrow
\hat{\rm H}(\mathscr{F}(\mathcal{W}))
$$
does not depend on the choice of $\chi$.
Moreover the map $\rho=\rho_\chi\circ\beta^{-1}_\chi$ is the inverse map of $b_\mathcal{W}$.
\end{thm}
In the sequel, we shall prove Theorem \ref{rho1}.
To see the first claim in Theorem \ref{rho1}, as $\mathcal{T}$ satisfies the condition (T-2), it suffices to show the claim only for the pair $\chi\prec\chi'$ in $\mathcal{T}$.
We construct the map
$$
\Theta^0:\mathscr{L}^{0}_{\chi}\underset{\mathbb{Z}_X}{\otimes} i_*or_{M/X}\longrightarrow\mathscr{L}^{0}_{\chi'}\underset{\mathbb{Z}_X}{\otimes} i_*or_{M/X}
$$
as follows.
Let us take a choice function $\psi$
$$
\psi:\Delta_{m-1}(\chi)\longrightarrow\Delta_{m-1}(\chi')
$$
such that $|\psi(\sigma)|\subset|\sigma|$ for any $\sigma\in\Delta_{m-1}(\chi)$.
Then we define the $\Theta^0$ by
\begin{equation*}
\Theta^0:c_\sigma\otimes\mathds{1}\mapsto <\sigma,\psi(\sigma)>c_\sigma|_{\psi(\sigma)}\otimes\mathds{1},
\end{equation*}
where $c_\sigma\in\mathbb{C}_{\overline{\sigma}}$ and $\mathds{1}\in i_*or_{M/X}$.
We have the following commutative diagram.
\begin{eqnarray*}
\begin{CD}
\mathscr{L}^{0}_{\chi}\underset{\mathbb{Z}_X}{\otimes} i_*or_{M/X} @>d^{0}_\chi>> \mathbb{C}_M\\
@V\Theta^0VV @VVidV \\
\mathscr{L}^{0}_{\chi'}\underset{\mathbb{Z}_X}{\otimes} i_*or_{M/X} @>>d^{0}_{\chi'}> \mathbb{C}_M.
\end{CD}
\end{eqnarray*}
As a matter of fact the commutativity of diagram above follows from
\begin{flalign*}
&(d^0\circ\Theta^0)(c_\sigma\otimes\mathds{1})
=\,d^0(<\sigma,\psi(\sigma)>c_\sigma|_{\psi(\sigma)}\otimes\mathds{1}) \\
=\,&c_\sigma|_{\psi(\sigma)}<\sigma,\psi(\sigma)>a_\mathds{1}(\psi(\sigma))
=\,c_\sigma a_\mathds{1}(\sigma)
=\,d^0(c_\sigma\otimes\mathds{1}),
\end{flalign*}
where $\sigma\in\Delta_{m-1}(\chi)$.

We can extend this commutative diagram to that of the complexes $\mathscr{L}^\bullet_\chi$ and $\mathscr{L}^\bullet_{\chi'}$ by the following proposition.
\begin{prop}\label{finer ex}
Let $\Theta^0$ be the map defined above.
There exist $\Theta^{k-m}(k=0,1,\cdots,m-1)$ such that the following diagram commutes.
\begin{eqnarray*}
\begin{CD}
\mathscr{L}^{-m}_\chi\underset{\mathbb{Z}_X}{\otimes} i_*or_{M/X} @>d^{-m}_\chi>>\cdots @>d^{-1}_\chi>> \mathscr{L}^{0}_\chi\underset{\mathbb{Z}_X}{\otimes} i_*or_{M/X} @>d^{0}_\chi>> \mathbb{C}_M @>>>0\\
@VV{\Theta^{-m}}V @. @VV{\Theta^{0}}V @VV{id}V  \\
\mathscr{L}^{-m}_{\chi'}\underset{\mathbb{Z}_X}{\otimes} i_*or_{M/X} @>d^{-m}_{\chi'}>>\cdots @>d^{-1}_{\chi'}>> \mathscr{L}^{0}_{\chi'}\underset{\mathbb{Z}_X}{\otimes} i_*or_{M/X} @>d^{0}_{\chi'}>> \mathbb{C}_M @>>>0.
\end{CD}
\end{eqnarray*}
\end{prop}
Since $X = M \times D^m$, it suffices to show Proposition \ref{finer ex} when $M$ is a point.
Hence, in what follows, we may assume that $M$ is contractible until the end of the proof of the proposition.

To show the proposition, we prepare for several lemmas.
Let $H_\chi$ be the family $\{H_i\}_{i\in\Lambda_\chi}$ of open half spaces $H_i$ in $\mathbb{R}^m$ passing through the origin which defines the stratification $\chi$, that is, $H_\chi$ is the minimal set satisfying that, for any $\sigma\in\Delta_{m-1}(\chi)$, $\sigma$ is given by the intersection of some $H_i$'s belonging to $H_\chi$ and $S^{m-1}$.

The most important property of the family $H_\chi$ is that;
for any $\sigma\in\Delta(\chi)$ and for any $H\in H_\chi$, $\sigma\cap H\neq\emptyset$ implies $\sigma\subset H$, and also $\sigma\cap\partial H\neq\emptyset$ implies $\sigma\subset\partial H$.
In subsequent arguments, this fact is constantly used.
\begin{lem}{\label{st}}
For any $\chi\in\mathcal{T}$ and $\sigma\in\Delta(\chi)$, the star open set ${\rm St}_\chi(|\sigma|)$ has the following expression.
\begin{equation}
{\rm St}_\chi(|\sigma|)=\left(\underset{\sigma\subset H_i\in H_\chi}{\cap}H_i\right)\cap S^{m-1}.
\end{equation}
\end{lem}
\begin{proof}
Firstly we prove ${\rm St}_\chi(|\sigma|)\subset\left(\underset{\sigma\subset H_i\in H_\chi}{\cap}H_i\right)\cap S^{m-1}$.
Let $x\in{\rm St}_\chi(|\sigma|)$.
By the definition of ${\rm St}_\chi(|\sigma|)$ there exists a cell $\tau$ such that $x\in\tau$ and $\sigma\subset\overline{\tau}$.
By considering the condition (T-1), there exist only two cases of the inclusion between $\tau$ and each $H_i$.
\begin{enumerate}
\item $\tau\subset (H_i)^c$.  \ Here we denote by $(\bullet)^c$ the complement of $(\bullet)$.
\item $\tau\subset H_i$.
\end{enumerate}
However the first case never occurs.
As a matter of fact, we have $\tau\subset (H_i)^c$ and $\sigma\subset\overline{\tau}$, that is, $\sigma\subset\overline{(H_i)^c}=(H_i)^c$, which contradicts $\sigma \subset H_i$.
Therefore we can ignore the first case and we obtain ${\rm St}_\chi(|\sigma|)\subset\left(\underset{\sigma\subset H_i\in H_\chi}{\cap}H_i\right)\cap S^{m-1}$.

Next we prove ${\rm St}_\chi(|\sigma|)\supset\left(\underset{\sigma\subset H_i\in H_\chi}{\cap}H_i\right)\cap S^{m-1}$.
It is sufficient to show the implication; $x\notin{\rm St}_\chi(|\sigma|) \Rightarrow x\notin\left(\underset{\sigma\subset H_i\in H_\chi}{\cap}H_i\right)\cap S^{m-1}$.
For any $x\notin{\rm St}_\chi(|\sigma|)$, there exists a cell $\delta$ with $x\in\delta$ and $\delta\notin{\rm St}_\chi(|\sigma|)$, that is, $\sigma\not\subset\overline{\delta}$ holds.
For such $\delta$ and $\sigma$ we can find an open half space $H\in H_\chi$ satisfying $\sigma\subset H$ and $\delta\subset (H)^c$.
In fact, this can be shown by the induction on $m$ in the following way;
When $m=0$, the claim clearly holds.
Then we assume that the claim holds for $m<\ell\ (\ell\geq1)$.
If $\delta\in\Delta_{\ell-1}(\chi)$, it follows from the definition of $\chi$ that $\underset{H\in H_{\chi,\delta}}{\cap}H=\delta$ holds, where $H_{\chi,\delta}$ denotes the set $\{H\in H_\chi\,;\,\delta\subset H\}$.
In particular, we have $\underset{H\in H_{\chi,\delta}}{\cap}\overline{H}=\overline{\delta}$.
Then we can find $H\in H_{\chi,\delta}$ with $\sigma\subset(\overline{H})^c$ because otherwise we get $\sigma\cap\overline{H}\neq\emptyset$ for any $H\in H_{\chi,\delta}$, which implies $\sigma\subset\underset{H\in H_{\chi,\delta}}{\cap}\overline{H}=\overline{\delta}$.
This contradicts the fact $\delta\notin{\rm St}_\chi(|\sigma|)$.

Now we assume $\delta\in\Delta_k(\chi)$ with $k<\ell-1$.
Then we can find $L\in H_{\chi,\delta}$ with $\delta\subset\partial L$.
If $\sigma\cap\partial L=\emptyset$, then clearly $L$ or $-L$ gives the desired half space $H$.
Therefore we may assume $\sigma\subset\partial L$ also.
Then, by applying the induction hypothesis to $\delta$ and $\sigma$ in $\partial L$,
we can find $H\in H_{\chi,\delta}$ such that $\delta\subset\partial L\cap\overline{H}$ and $\sigma\subset\partial L\cap(\overline{H})^c$.
Hence we have obtained the claim.

Therefore the converse implication follows and this completes the proof.
\end{proof}
The following lemma is the crucial key for the proof of Proposition \ref{finer ex}.
\begin{lem}{\label{UV}}
Assume that $M$ is contractible.
Then $M*{\rm St}_{\chi'} (|\tau|)\setminus M*{\rm St}_\chi (|\sigma|)$ is contractible for any $\sigma\in\Delta(\chi)$ and $\tau\in\Delta(\chi')$.
\end{lem}
\begin{rmk}
The assumption $\chi\prec\chi'$ is essential, otherwise the claim does not hold.
\end{rmk}
\begin{proof}
We set $U=M*{\rm St}_\chi (|\sigma|)$ and $V=M*{\rm St}_{\chi'} (|\tau|)$.
We also set $H_{\chi,\sigma}=\{H_i\in H_\chi\ ;\ \sigma\subset H_i\}$.
By Lemma \ref{st} we have
\begin{flalign*}
V\setminus U
&=M*\left\{{\rm St}_{\chi'}(|\tau|) \setminus (\underset{H_i\in H_{\chi,\sigma}}{\cap}\,H_i)\cap S^{m-1}\right\} \\
&=\underset{H_i\in H_{\chi,\sigma}}{\cup}\ M*({\rm St}_{\chi'}(|\tau|)\setminus H_i\cap S^{m-1}).
\end{flalign*}
Therefore we just prove that $\underset{H_i\in H_{\chi,\sigma}}{\cup}\ M*({\rm St}_{\chi'}(|\tau|)\setminus H_i\cap S^{m-1})$ is contractible.
By considering $\chi\prec\chi'$ there are three cases of the inclusion between $\tau$ and each $H_i$ containing $\sigma$.
\begin{enumerate}
\item The case of $\tau\subset H_i\cap S^{m-1}$.

Since $\chi'$ is finer than $\chi$ we have ${\rm St}_{\chi'}(|\tau|)\subset H_i$, that is, $M*({\rm St}_{\chi'}(|\tau|)\setminus H_i\cap S^{m-1})=\emptyset$ and we can ignore this case.

\item The case of $\tau\subset -H_i\cap S^{m-1}$.

We have $M*({\rm St}_{\chi'}(|\tau|)\setminus H_i\cap S^{m-1})=M*{\rm St}_{\chi'}(|\tau|)=V$.
Therefore if such an $H_i$ exists, then $V\setminus U=V$, and hence it becomes contractible.

\item The case of $\tau\subset\partial H_i\cap S^{m-1}$.
\end{enumerate}
By the observations in the cases 1 and 2, we may assume that all the $H_i$ containing $\sigma$ satisfy the case 3.
We first show the claim that $\underset{H_i\in H_{\chi,\sigma}}{\cap}\,(H_i)^c\,\cap S^{m-1}$ has an interior point.
We take points $x\in\sigma$ and $y\in\tau$.
Let $z$ be a point in the line passing through $x$ and $y$ so that the points $x,y,z$ are situated in this order.
Then we can easily see, by noticing $\sigma=\underset{H_i\in H_{\chi,\sigma}}{\cap}H_i$ and $y\in H_i\ (H_i\in H_{\chi,\sigma})$, that $z$ becomes an interior point of $\underset{H_i\in H_{\chi,\sigma}}{\cap}\,(H_i)^c\,\cap S^{m-1}$.
By choosing $z$ in close enough to the point $y$, $z$ belongs to ${\rm St}_{\chi'}(|\tau|)$ and hence we have
$z\in{\rm St}_{\chi'}(|\tau|) \cap \{\underset{H_i\in H_{\chi,\sigma}}{\cap}\,(H_i)^c\cap S^{m-1}\}$.
Finally, using $z$, we prove the contractibleness of $V\setminus U$ and complete the proof.
We have
\begin{flalign*}
&{\rm St}_{\chi'}(|\tau|) \cap \{\underset{H_i\in H_{\chi,\sigma}}{\cap}\,(H_i)^c\cap S^{m-1}\} \\
=\,&\underset{H_i\in H_{\chi,\sigma}}{\cap}\ {\rm St}_{\chi'}(|\tau|)\setminus (H_i\cap S^{m-1}).
\end{flalign*}
Because each ${\rm St}_{\chi'}(|\tau|)\setminus H_i\cap S^{m-1}$ is a convex set,
$\underset{H_i\in H_{\chi,\sigma}}{\cup}\ M*({\rm St}_{\chi'}(|\tau|)\setminus H_i\cap S^{m-1})$ is contractible to the point $z\in\underset{H_i\in H_{\chi,\sigma}}{\cap}\ {\rm St}(|\tau|)_{\chi'}\setminus H_i\cap S^{m-1}$.
This complete the proof.
\end{proof}
\begin{lem}{\label{exlem}}
Assume $\chi\prec\chi'$ and $M$ to be contractible.
Then for any $\sigma\in\Delta(\chi)$ and for any $\tau\in\Delta(\chi')$ we have
\begin{equation}{\label{Ext}}
{\rm Ext}^\ell(\mathbb{C}_{\overline{\sigma}},\mathbb{C}_{\overline{\tau}})=0 \quad (\ell\neq 0).
\end{equation}
\end{lem}
\begin{proof}
We set $U=M*{\rm St}_\chi (|\sigma|)$ and $V=M*{\rm St}_{\chi'} (|\tau|)$.
Then we obtain
\begin{eqnarray}
{\rm Ext}^\ell(\mathbb{C}_{\overline{\sigma}},\mathbb{C}_{\overline{\tau}})
={\rm H}^\ell({\rm RHom}(\mathbb{C}_{\overline{U}},\mathbb{C}_{\overline{V}}))
\simeq{\rm H}_{\overline{U}\cap \overline{V}}^\ell(\overline{V},\mathbb{C}_{\overline{V}}).
\end{eqnarray}
Here we have the following long exact sequence.
\begin{eqnarray}{\label{coh-sequence}}
\begin{CD}
0@>>>{\rm H}_{\overline{U}\cap \overline{V}}^0(\overline{V},\mathbb{C}_{\overline{V}}) @>>> {\rm H}^0(\overline{V},\mathbb{C}_{\overline{V}}) @>>> {\rm H}^0(\overline{V}\setminus\overline{U},\mathbb{C}_{\overline{V}}) \\ 
\\
@>>>{\rm H}_{\overline{U}\cap \overline{V}}^1(\overline{V},\mathbb{C}_{\overline{V}}) @>>> {\rm H}^1(\overline{V},\mathbb{C}_{\overline{V}})  @>>>\cdots.
\end{CD}
\end{eqnarray}
By Lemma \ref{st}, $V$ is contractible and we obtain
\begin{eqnarray*}
{\rm H}^k(\overline{V},\mathbb{C}_{\overline{V}}) \simeq \begin{cases}
    \mathbb{C} & (k=0), \\
    0 & (k\neq 0).
  \end{cases}
\end{eqnarray*}
Moreover, by Lemma \ref{UV} we have
\begin{eqnarray*}
{\rm H}^k(\overline{V}\setminus \overline{U},\mathbb{C}_{\overline{V}}) \simeq \begin{cases}
    \mathbb{C} & (k=0), \\
    0 & (k\neq 0).
  \end{cases}
\end{eqnarray*}
By \eqref{coh-sequence}, ${\rm H}_{\overline{U}\cap \overline{V}}^k(\overline{V},\mathbb{C}_{\overline{V}})=0$ holds for any $k\neq 0,1$.
Moreover we have the exact sequence
\begin{equation*}
0\rightarrow{\rm H}_{\overline{U}\cap \overline{V}}^0(\overline{V},\mathbb{C}_{\overline{V}}) \rightarrow \mathbb{C} \rightarrow \mathbb{C} \rightarrow {\rm H}_{\overline{U}\cap \overline{V}}^1(\overline{V},\mathbb{C}_{\overline{V}}) \rightarrow0.
\end{equation*}
From which we obtain ${\rm H}_{\overline{U}\cap \overline{V}}^0(\overline{V},\mathbb{C}_{\overline{V}})=0$ and ${\rm H}_{\overline{U}\cap \overline{V}}^1(\overline{V},\mathbb{C}_{\overline{V}})=0$.
The proof of Lemma \ref{exlem} has been finished.
\end{proof}

Now we are ready to prove Proposition \ref{finer ex}.
\begin{proof}
We will show existence of $\Theta^k$ under the assumption of that of $\Theta^{k+1}$.
Then by the induction we obtain all the $\Theta^k\,(k=-m,-m+1,\cdots,-1)$.
Assume that we construct $\Theta^{k+1}$.
By \eqref{ex-sequence}, we obtain the following exact sequence for $k=-m,-m+1,\cdots,-1$.
\begin{eqnarray}{\label{sequence}}
0\longrightarrow
{\rm Ker}\,d^k_{\chi'}
\overset{id}{\longrightarrow}
\mathscr{L}_{\chi'}^k\underset{\mathbb{Z}_X}{\otimes} i_*or_{M/X}\overset{d_\chi^k}{\longrightarrow}
{\rm Im}\,d^k_{\chi'}
\longrightarrow0.
\end{eqnarray}
The following long exact sequence is induced from \eqref{sequence}.
\begin{eqnarray}{\label{longs}}
\begin{CD}
0@>>>{\rm Hom}(\mathscr{R}^k_\chi,{\rm Ker}\, d^k_{\chi'}) @>>> {\rm Hom}(\mathscr{R}^k_\chi,\mathscr{R}^k_{\chi'}) @>>> {\rm Hom}(\mathscr{R}^k_\chi,{\rm Im}\, d^k_{\chi'}) \\ 
\\
@>>>{\rm Ext}^1(\mathscr{R}^k_\chi,{\rm Ker}\, d^k_{\chi'}) @>>> {\rm Ext}^1(\mathscr{R}^k_\chi,\mathscr{R}^k_{\chi'})  @>>> \cdots,
\end{CD}
\end{eqnarray}
where $\mathscr{R}^k_\chi=\mathscr{L}_{\chi}^k\underset{\mathbb{Z}_X}{\otimes} i_*or_{M/X}$.
Hence it is enough to show ${\rm Ext}^1(\mathscr{R}_{\chi}^k,{\rm Ker}\, d^k_{\chi'})=0$ for the existence of $\Theta_k\in {\rm Hom}(\mathscr{R}_{\chi}^k,\mathscr{R}_{\chi'}^k)$.
By the definition of $\mathscr{R}_{\chi}^k$ we have
\begin{align*}
{\rm Ext}^1(\mathscr{R}_{\chi}^k,{\rm Im}\, d^k_{\chi'})
=&{\rm Ext}^1\left(\left(\underset{\sigma\in\Delta_{k+m-1}(\chi)}{\oplus}\mathbb{C}_{\overline{\sigma}}\right)\underset{\mathbb{Z}_X}{\otimes} i_*or_{M/X},{\rm Ker}\, d^k_{\chi'}\right)\\
\simeq&\underset{\sigma\in\Delta_{k+m-1}(\chi)}{\oplus}{\rm Ext}^1(\mathbb{C}_{\overline{\sigma}}\underset{\mathbb{Z}_X}{\otimes} i_*or_{M/X},{\rm Ker}\, d^k_{\chi'})\\
\simeq&\underset{\sigma\in\Delta_{k+m-1}(\chi)}{\oplus}{\rm Ext}^1(\mathbb{C}_{\overline{\sigma}},{\rm Ker}\, d^k_{\chi'})\underset{\mathbb{Z}_X}{\otimes} i_*or_{M/X}.
\end{align*}
Therefore it is sufficient to prove the following equation for any $k=-m,-m+1,\cdots,-1$ and for any $\sigma\in\Delta(\chi)$.
\begin{equation}\label{ext1}
{\rm Ext}^1(\mathbb{C}_{\overline{\sigma}},{\rm Ker}\, d^k_{\chi'})=0.
\end{equation}
By taking the short exact sequence below into account
$$
0\longrightarrow
{\rm Ker}\,d^{k-1}_{\chi'}
\overset{id}{\longrightarrow}
\mathscr{R}_{\chi'}^{k-1}\overset{d^{k-1}_{\chi'}}{\longrightarrow}
{\rm Im}\,d^{k-1}_{\chi'}={\rm Ker}\,d^k_{\chi'}
\longrightarrow0,
$$
we can obtain the following long exact sequence.
\begin{eqnarray}{\label{longs}}
\begin{CD}
0@>>>{\rm Hom}(\mathbb{C}_{\overline{\sigma}},{\rm Ker}\, d^{k-1}_{\chi'}) @>>> {\rm Hom}(\mathbb{C}_{\overline{\sigma}},\mathscr{R}^{k-1}_{\chi'}) @>>> {\rm Hom}(\mathbb{C}_{\overline{\sigma}},{\rm Ker}\, d^k_{\chi'}) \\ 
\\
@>>>{\rm Ext}^1(\mathbb{C}_{\overline{\sigma}},{\rm Ker}\, d^{k-1}_{\chi'}) @>>> {\rm Ext}^1(\mathbb{C}_{\overline{\sigma}},\mathscr{R}^{k-1}_{\chi'})  @>>> \cdots.
\end{CD}
\end{eqnarray}
Let us prove the following equation for any $\ell>0$.
\begin{equation}{\label{ext}}
{\rm Ext}^\ell(\mathbb{C}_{\overline{\sigma}},\mathscr{R}^{k-1}_{\chi'})=0.
\end{equation}\label{extcs}
By remembering the definition of $\mathscr{R}^{k}_{\chi'}$, we can calculate ${\rm Ext}^l(\mathbb{C}_{\overline{M*{\rm St}(|\sigma|)}},\mathscr{R}^{k}_{\chi'})$ as follows.
\begin{align*}
{\rm Ext}^\ell(\mathbb{C}_{\overline{\sigma}},\mathscr{R}^{k}_{\chi'})
=&{\rm Ext}^\ell\left(\mathbb{C}_{\overline{\sigma}},\underset{\tau\in\Delta_{m+k-1}(\chi')}{\oplus}\mathbb{C}_{\overline{\tau}}\underset{\mathbb{Z}_X}{\otimes} i_*or_{M/X}\right)\\
\simeq&\underset{\tau\in\Delta_{m+k-1}(\chi')}{\oplus}{\rm Ext}^\ell \left(\mathbb{C}_{\overline{\sigma}},\mathbb{C}_{\overline{\tau}}\right)\underset{\mathbb{Z}_X}{\otimes} i_*or_{M/X}\\
=&0.
\end{align*}
The last equation follows from Lemma \ref{UV}.
By \eqref{longs} and \eqref{ext} we obtain
$$
{\rm Ext}^\ell(\mathbb{C}_{\overline{\sigma}},{\rm Ker}\, d^k_{\chi'})
\simeq{\rm Ext}^{\ell+1}(\mathbb{C}_{\overline{\sigma}},{\rm Ker}\, d^{k-1}_{\chi'}).
$$
This means that
\begin{equation*}
{\rm Ext}^\ell(\mathbb{C}_{\overline{\sigma}},{\rm Ker}\, d^k_{\chi'})
\simeq{\rm Ext}^{\ell+k+m}(\mathbb{C}_{\overline{\sigma}},{\rm Ker}\, d^{-m}_{\chi'})
=0
\end{equation*}
because of ${\rm Ker}\,d^{-m}_{\chi'}=0$.
Hence we have shown \eqref{ext1} and the proof of Proposition \ref{finer ex} has been finished.
\end{proof}

By Proposition \ref{finer ex}, the following diagram of complexes commutes for $\chi\prec\chi'$.
\begin{eqnarray}
\begin{CD}
\mathscr{L}^{\bullet}_{\chi}\underset{\mathbb{Z}_X}{\otimes} i_*or_{M/X} @>\sim>> \mathbb{C}_M\\
@V\Theta^{\bullet}VV @| \\
\mathscr{L}^{\bullet}_{\chi'}\underset{\mathbb{Z}_X}{\otimes} i_*or_{M/X} @>\sim>> \mathbb{C}_M.
\end{CD}
\end{eqnarray}
By applying ${\rm RHom}_{\mathbb{C}_X}(D(\bullet)\otimes\mathbb{C}_U,\mathscr{F})$ to the above diagram,
taking the $0$-th cohomology groups and taking inductive limit $\ilim[U\supset M]$ of these cohomology groups, we obtain
\begin{eqnarray*}
\begin{CD}
\hmx @>\sim>\beta^{-1}_\chi> \dfrac{\mathscr{F}_m(\chi)}
{{\rm Im}(\mathscr{F}_{m-1}(\chi)\rightarrow\mathscr{F}_m(\chi))}
\underset{\mathbb{Z}_M(M)}{\otimes}or_{M/X}(M) \\
@| @VVA^0(\Theta)V \\
\hmx @>\sim>\beta^{-1}_{\chi'}> \dfrac{\mathscr{F}_m(\chi')}
{{\rm Im}(\mathscr{F}_{m-1}(\chi')\rightarrow\mathscr{F}_m(\chi'))}
\underset{\mathbb{Z}_M(M)}{\otimes}or_{M/X}(M),
\end{CD}
\end{eqnarray*}
where $A^0(\Theta)=\ilim[U\in\mathcal{U}]{\rm H}^0({\rm RHom}_{\mathbb{C}_X}(D(\Theta^\bullet)\otimes\mathbb{C}_U,\mathscr{F}))$.
Moreover by the calculation
\begin{flalign*}
\rho_{\chi'} \circ A^0(\Theta)(f_\sigma\otimes\mathds{1})
=&\rho_{\chi'}(<\sigma,\psi(\sigma)>f_\sigma|_{\psi(\sigma)}\otimes\mathds{1}) \\
=&<\sigma,\psi(\sigma)>a_\mathds{1}(\psi(\sigma))f_\sigma|_{\psi(\sigma)} \\
=&a_\mathds{1}(\sigma)f_\sigma|_{\psi(\sigma)}
=a_\mathds{1}(\sigma)f_\sigma
=\rho_\chi (f_\sigma\otimes\mathds{1}),
\end{flalign*}
we have the following commutative diagram.
\begin{eqnarray}\label{Fhat}
\begin{CD}
\dfrac{\mathscr{F}_m(\chi)}
{{\rm Im}(\mathscr{F}_{m-1}(\chi)\rightarrow\mathscr{F}_m(\chi))}
\underset{\mathbb{Z}_M(M)}{\otimes}or_{M/X}(M) @>\sim>\rho_\chi> \hat{\rm H}(\mathscr{F}(\mathcal{W})) \\
@VA^0(\Theta)VV @| \\
\dfrac{\mathscr{F}_m(\chi')}
{{\rm Im}(\mathscr{F}_{m-1}(\chi')\rightarrow\mathscr{F}_m(\chi'))}
\underset{\mathbb{Z}_M(M)}{\otimes}or_{M/X}(M) @>\sim>\rho_{\chi'}> \hat{\rm H}(\mathscr{F}(\mathcal{W})).
\end{CD}
\end{eqnarray}
Hence we get
\begin{flalign*}
\rho_\chi\circ\beta^{-1}_\chi
=\rho_{\chi'} \circ A^0(\Theta)\circ\beta^{-1}_{\chi}
=\rho_{\chi'} \circ \beta^{-1}_{\chi'},
\end{flalign*}
and we have seen that $\rho=\rho_\chi\circ\beta^{-1}_\chi$ does not depend on the choices of $\chi$.

We finally show $b_\mathcal{W}\circ\rho=id$ and $\rho\circ b_\mathcal{W}=id$.
The formula $b_\mathcal{W}\circ\rho=id$ follows from the commutativity of \eqref{fivecom}.

To show $\rho\circ b_\mathcal{W}=id$, we may fix $\chi\in\mathcal{T}$ and $W\in\mathcal{W}$, and take $f\in\mathscr{F}(W)$.
By Proposition \ref{main1} and Definition \ref{rho} we have
\begin{flalign}
\rho\circ b_\mathcal{W}(f)
=\rho_\chi \circ \beta_\chi^{-1} \circ b_\mathcal{W} (f \otimes \mathds{1}_\sigma)
=\rho_\chi(f a_\mathds{1}(\sigma)\otimes \mathds{1}_\sigma)
=(a_\mathds{1}(\sigma))^2 f= f.
\end{flalign}
These complete the proof of Theorem \ref{mainthm}.
\section{Application to Laplace hyperfunctions}
In the last section, we shall introduce intuitive representation of Laplace hyperfunctions as an application of Theorem \ref{mainthm}. 

First we recall the sheaf of Laplace hyperfunctions.
Let $n$ be a natural number and $S^{2n-1}$ the unit sphere in $\mathbb{C}^n \simeq \mathbb R^{2n}$.
\begin{dfn}
We define the radial compactification $\mathbb D_{\mathbb{C}^n}$ of $\mathbb{C}^n$ as follows.
\begin{equation}
\rcsp=\mathbb{C}^n \sqcup S^{2n-1}\infty.
\end{equation}
\end{dfn}
In the same way we define
\begin{equation}
\mathbb{D}_{\mathbb{R}^n}=\mathbb{R}^n \sqcup S^{n-1}\infty.
\end{equation}
A family of fundamental neighborhoods of $z_0\in \mathbb{C}^n$ consists of 
\begin{equation}
B_\varepsilon(z_0)=\{ z\in\mathbb{C}^n\,;\,|z-z_0|<\varepsilon \}
\end{equation}
for $\varepsilon>0$, and that of $\zeta_0\infty\in S^{2n-1}\infty$ consists of an open cone

\begin{equation}
G_r(\Gamma)=\left\{z\in \mathbb{C}^n\,;\,|z|>r,\frac{z}{|z|}\in\Gamma\right\}\cup\Gamma\infty,
\end{equation}
where $r>0$ and $\Gamma$ runs through open neighborhoods of $\zeta_0$ in $S^{2n-1}$.
We define the set of holomorphic functions of exponential type on $\mathbb{D}_{\mathbb{C}^n}$.
\begin{dfn}
Let $U$ be an open subset of $\rcsp$. We define the set $\hexpo(U)$ of holomorphic functions of exponential type on $U$ as follows.
A holomorphic function $f(z)$ on $U\cap\mathbb{C}^n$ belongs to $\hexpo(U)$ if, for any compact set $K\subset U$, there exist positive constants $C_K$ and $H_K$ such that
\begin{eqnarray}
|f(z)|\leq C_Ke^{H_K|z|}\hspace{40pt} (z\in K\cap\mathbb{C}^n).
\end{eqnarray}
Moreover we denote by $\hexpo$ the sheaf of holomorphic functions of exponential type on $\rcsp$.
\end{dfn}

To construct intuitive representation of Laplace hyperfunctions, we need the vanishing theorem of the global cohomology groups on a Stein open set with coefficients in $\hexpo$.
However, by Umeta and Honda (\cite{Honda}), we cannot expect such a vanishing theorem generally.
To overcome this difficulty, they introduce a new property called regular at $\infty$ for a subset in $\mathbb{D}_{\mathbb{C}^n}$.

Let $U$ be a subset in $\rcsp$.
We define the closed set ${\rm clos}_\infty^1 (U) \subset S^{2n-1}\infty$ as follows.
A point $z\in S^{2n-1}\infty$ belongs to ${\rm clos}_\infty^1 (U)$ if there exists a point sequence $\{z_k\}_{k\in\mathbb{N}}$ in $U\cap \mathbb{C}^n$ such that

\begin{eqnarray}
	z_k\rightarrow z\, \text{and}\,
\frac{|z_{k+1}|}{|z_k|}\rightarrow 1\ \text{in}\ \rcsp\ (k\rightarrow\infty).
\end{eqnarray}
Set 
\begin{equation}
N^1_\infty(U)=S^{2n-1}\infty\setminus{\rm clos}^1_\infty(\mathbb{C}^n \setminus U).
\end{equation}

\begin{dfn}[\cite{Honda}]
Let $U$ be an open set in $\rcsp$.
If $N^1_\infty(U)=U\cap S^{2n-1}\infty$ holds, we say that $U$ is regular at $\infty$.
\end{dfn}
Then we have the following theorem.
\begin{thm}[\cite{Honda}]{\label{th:expo-vanishing}}
Let $U$ be an open set in $\rcsp$.
Assume that $U$ is regular at $\infty$ and $U \cap \mathbb{C}^n$ is a Stein open set in $\mathbb{C}^n$.
Then we obtain
\begin{equation}
\operatorname{H}^k(U,\, \hexpo) = 0 \qquad (k \neq 0).
\end{equation}
\end{thm}
As a similar notion, we define $N_\infty(U) \subset S^{2n-1}\infty$ for a subset $U$ in $\mathbb{D}_{\mathbb{C}^n}$ as
\begin{equation}
	N_\infty(U) = S^{2n-1}\infty\setminus\overline{\left( \mathbb{C}^n \setminus U \right)},
\end{equation}
where the closure $\overline{(\bullet)}$ is taken in $\rcsp$.
Note that $N^1_\infty(U) \subset N_\infty(U)$ holds.
\begin{dfn}
For an open set $A\subset\mathbb{C}^n$, we define
\begin{equation}
\widehat{A}=A\cup N_\infty(A).
\end{equation}
\end{dfn}
Remark that $\widehat{A}$ is the largest open set in $\rcsp$ such that $\widehat{A} \cap \mathbb{C}^n = A$ holds.
We also define the similar set $\widehat{B}\subset\mathbb{D}_{\mathbb{R}^n}$ for a set $B\subset\mathbb{R}^n$.
Sometimes we write $\widehat{}\,\bullet$ instead of $\widehat{\bullet}$.
\begin{dfn}[\cite{Honda}]
The sheaf $\mathscr{B}^{\rm exp}_{\mathbb{D}_{\mathbb{R}^n}}$ of Laplace hyperfunctions on $\mathbb{D}_{\mathbb{R}^n}$ is defined as follows.
\begin{eqnarray}
\mathscr{B}^{\rm exp}_{\mathbb{D}_{\mathbb{R}^n}}=
\mathscr{H}^n_{\mathbb{D}_{\mathbb{R}^n}}(\hexpo)\otimes 
or_{\mathbb{D}_{\mathbb{R}^n}/\mathbb{D}_{\mathbb{C}^n}},
\end{eqnarray}
where $or_{\mathbb{D}_{\mathbb{R}^n}/\mathbb{D}_{\mathbb{C}^n}}$ is the orientation sheaf
$\mathscr{H}^n_{\mathbb{D}_{\mathbb{R}^n}}(\mathbb{Z}_{\rcsp})$.
\end{dfn}

Let $\Omega$ be an open cone in $\mathbb{R}^n$ with vertex $a\in\mathbb{R}^n$.
Note that $\Omega$ is not assumed to be convex.
We define an open set $X$ in $\mathbb{C}^n$ as follows.
\begin{equation}
X =\{z = x+\sqrt{-1}y \in \mathbb{C}^n;\, x \in \Omega,\, |y|^2 < |x-a|^2 + 1\}.
\end{equation}
Note that $\widehat{\Omega} = \widehat{X} \cap \mathbb{D}_{\mathbb{R}^n}$.

We define intuitive representation of Laplace hyperfunctions $\bexpo(\widehat{\Omega})$ on $\widehat{\Omega}$.
First we define the family $\mathcal{W}$ of open sets in $\widehat{X}$.
Let $\Gamma$ be an $\mathbb{R}_+$-conic connected open set in $\mathbb{R}^{n}$.

\begin{dfn}
Let $W$ be an open set in $\widehat{X}$.
We say that $W$ is an infinitisimal wedge of type $\Omega\times\sqrt{-1}\Gamma$ if and only if the following conditions hold.
\begin{enumerate}
\item $W \subset \widehat{}\,(\Omega \times \sqrt{-1}\Gamma)$.
\item For any open proper subcone $\Gamma'$ of $\Gamma$, there exists an open neighborhood $U \subset \widehat{X}$ of $\widehat{\Omega}$ such that
\begin{equation}
\widehat{}\,(\Omega \times \sqrt{-1}\Gamma') \cap U \subset W.
\end{equation}
\end{enumerate}
\end{dfn}
Let $\mathcal{W}(\sqrt{-1}\Gamma)$ be a family of all infinitisimal wedges of type $\Omega\times\sqrt{-1}\Gamma$ in $\widehat{X}$, and we set
\begin{equation}
\mathcal{W} = \underset{\Gamma}{\bigcup}\,\, \mathcal{W}(\sqrt{-1}\Gamma).
\end{equation}
Here $\Gamma$ runs through all the $\mathbb{R}_+$-conic connected open set in $\mathbb{R}^n$.

\begin{thm}
For aforementioned $\mathcal{W}$, there exists $(\mathcal{T},\mathcal{U},\,\mathcal{W})$ which satisfies the conditions (T), (U) and (W), and we have intuitive representation of Laplace hyperfunctions $\bexpo(\widehat{\Omega})$ by
\begin{eqnarray}
\hat{\rm H}(\hexpo(\mathcal{W}))=\left(\underset{W \in \mathcal{W}}{\oplus} \hexpo(W) \right)/\mathcal{R},
\end{eqnarray}
where $\mathcal{R}$ is a $\mathbb{C}$-vector space generated by the following elements.
\begin{eqnarray}
f\,{\oplus}\,(-f|_{W_2})\quad (f\in\hexpo(W_1)).
\end{eqnarray}
Here $W_1$ and $W_2$ are open sets in $\mathcal{W}$ with $W_2\subset W_1$.
\end{thm}

To prove this theorem, we construct $(\mathcal{T},\mathcal{U},\mathcal{W})$ satisfying the condition (T), (U) and (W).
First we show the theorem of Grauert type on $\rcsp$.

\begin{thm}{\label{th:laplace-grauert}}
Let $\Omega$ be an $\mathbb{R}_+$-conic open set in $\mathbb{R}^n$ and let $V$ be an open neighborhood of $\widehat{\Omega}$ in $\rcsp$.
Then we can find an open set $U$ in $\rcsp$ satisfying the following three conditions.
\begin{enumerate}
\item $\widehat{\Omega} \subset U \subset V$.
\item $U$ is regular at $\infty$.
\item $U \cap \mathbb{C}^n$ is a Stein open set.
\end{enumerate}
\end{thm}

\begin{proof}
Retaking $V$ smaller we can assume the following conditions.
\begin{enumerate}
\renewcommand{\labelenumi}{(\alph{enumi})}
\item $V \cap \mathbb{R}^n = \Omega$.

\item $V \cap \mathbb{C}^n 
	\subset \{z = x + \sqrt{-1}y \in \mathbb{C}^n\,;\, |y| \le\mbox{max}\{1, |x|/2\}\}$.

\item $\widehat{V \cap \mathbb{C}^n} = V$.
\end{enumerate}

We introduce some new functions. Let $r(x)$ be a function on $\mathbb{R}^n$ such that
\begin{equation}
r(x) = \operatorname{dist}(x,\, 
(\mathbb{C}^n \setminus V) \cap (\{x\} \times \sqrt{-1}\mathbb{R}^n)) \qquad
(x \in \mathbb{R}^n).
\end{equation}
We also define the function $r^*(\omega)$ on $S^{n-1}$ by
\begin{equation}
r^*(\omega) = \inf_{t \ge 1} \dfrac{r(t\omega)}{t} \qquad
(\omega \in S^{n-1}).
\end{equation}
Here $r^*$ has the following properties.
\begin{enumerate}
\renewcommand{\labelenumi}{(\roman{enumi})}
\item $r^*(\omega)$ is a lower semi-continuous function on $S^{n-1}$.
\item $r^*(\omega) >0$ for $(\omega \in \Omega \cap S^{n-1})$ and $r^*(\omega) = 0$ otherwise.

\end{enumerate}

For $x^* \in \mathbb{R}^n$ and for $c>0$, we define an analytic polyhedron $U(x^*, c)$ as follows.
\begin{equation}
U(x^*, c)=\{ z = x + \sqrt{-1}y \in \mathbb{C}^n\,;\, |y|^2 < |x - x^*|^2 + c^2 \}.
\end{equation}
Moreover we define an open set $U_1$ in $\mathbb{C}^n$ as the interior of the following set.
\begin{equation}
\bigcap_{x^* \in \mathbb{R}^n,\,x^* \ne 0}
W\left(x^*,\, \dfrac{|x^*| r^*(x^*/|x^*|)}{2}\right).
\end{equation}
Then $U_1$ has the following properties.
\begin{enumerate}
\renewcommand{\labelenumi}{(\arabic{enumi})}
\item It is a $\mathbb{R}_+$-conic set.
\item It is a Stein open set.
\item $\Omega \subset U_1$ and $U_1\cap \{z = x + \sqrt{-1}y \in \mathbb{C}^n\,;\, |x| \ge 1\}\subset V$ hold.
\end{enumerate}
These properties are shown in the following way:

We first show (1).
Let $z=x+\sqrt{-1}y\in U_1$.
For any $x^*\in\mathbb{R}^n\setminus \{0\}$, $z\in U_{x^*}:=U\left(x^*,\, \dfrac{|x^*| r^*(x^*/|x^*|)}{2}\right)$ holds.
Hence for any $t\in\mathbb{R}_+$ and for any $x^*\in\mathbb{R}^n\setminus \{0\}$ it is sufficient to prove $tz\in U_{x^*}$.
Since $z\in U_{x^*}$ holds for any $x^*\in\mathbb{R}^n\setminus \{0\}$, by the definition we get
\begin{equation}
|y|^2<|x-x^*|^2+\left(\dfrac{|x^*| r^*(x^*/|x^*|)}{2}\right)^2.
\end{equation}
By multiplying the both sides by $t^2$
\begin{equation}
|ty|^2<|tx-tx^*|^2+\left(\dfrac{|tx^*| r^*(tx^*/|tx^*|)}{2}\right)^2.
\end{equation}
Furthermore, for any $t\in\mathbb{R}_+$ and for any $x^*\in\Omega$, $tx^*\in\Omega$ holds.
Hence we have $tz\in U_1$ and we have shown (1).

Next we give the proof of (2).
Because
\begin{flalign}
&U\left(x^*,\, \dfrac{|x^*| r^*(x^*/|x^*|)}{2}\right) \nonumber \\
=&\left\{z\in\mathbb{C}^n\,;\,\left|{\rm exp}\left(-(z-x^*)^2-\left(\dfrac{|x^*| r^*(x^*/|x^*|)}{2}\right)^2\right)\right|<1\right\},
\end{flalign}
$U_1$ is a ${\rm Stein}$ open set.
The proof of (2) has been completed.

Finally we show (3).
$\Omega\subset U_1$ is obvious from the construction of $U_1$.
Moreover it is easy to see $U_1\cap \{z = x + \sqrt{-1}y \in \mathbb{C}^n\,;\, |x| \ge 1\}\subset V$ by the definition of $r^*(\omega)$.
Hence (3) follows.
Let $\widetilde{U_1}$ be the interior of
$$
\bigcap_{x^* \in \mathbb{R}^n,\,x^* \geq 1}
U\left(x^*,\, \dfrac{|x^*| r^*(x^*/|x^*|)}{2}\right).
$$
Let us define $U_2$ by the interior of
\begin{equation}
U_2=
\bigcap_{x^* \in \mathbb{R}^n,\,|x^*| \le 1}
U\left(x^*,\, r(x^*)/2\right).
\end{equation}
Since there exists $R>0$ such that
$$
\widetilde{U_1}=U_1\quad\mbox{in }\{z=x+\sqrt{-1}y\,;\,|x|>R\}
$$
we see that $\widetilde{U_1}$ is regular at $\infty$ and a Stein open set.
Then $U = \widehat{}\,{(\widetilde{U_1} \cap U_2)}$ satisfies the conditions of the theorem.
\end{proof}

Set $M = \widehat{\Omega}$. 
We define a homeomorphism $\iota: \widehat{X}\cap\mathbb{C}^n \to M \times D^n$ as follows.
\begin{equation}
	\iota(x + \sqrt{-1}y) = \left(x,\, \dfrac{y}{\sqrt{|x|^2 + 1}}\right).
\end{equation}
By extending the morphism to $\widehat{X}$ continuously, we obtain the homeomorphism between $\widehat{X}$ and
$\widehat{\Omega} \times D^n = M \times D^n$.

\

Let $\mathcal{U}$ be a family of all the open subsets $U$ in $\widehat {X}$ which satisfy
\begin{enumerate}
\item $U$ is a neighborhood of $M$.
\item $U$ is regular at $\infty$.
\item $U\cap \mathbb{C}^n$ is a Stein open set.
\end{enumerate}
Such an open subset surely exists thanks to Theorem \ref{th:laplace-grauert}.

\

Let $\mathcal{T}$ be a family of all the stratifications of $S^{n-1}$ which are associated with partitioning by a finite family of hyperplanes passing through the origin in $\mathbb{R}^n$.
It is clear that $\mathcal{T}$ satisfies the condition (T).

\

On the aforementioned preparations, we can easily prove that $(\mathcal{T},\,\mathcal{U},\,\mathcal{W})$ satisfies the conditions (T), (U) and (W) by applying Theorem \ref{th:expo-vanishing} and Theorem \ref{th:laplace-grauert}, and thus intuitive representation of Laplace hyperfunctions is justified.

\end{document}